\documentclass{article}

\usepackage{amsmath, amsthm, amssymb, mathrsfs, algorithm}
\usepackage{graphicx, subfigure, tikz}
\usepackage{hyperref} 

\graphicspath{{Figure/}}

\newtheorem{example}{\noindent Example}
\newtheorem{theorem}{Theorem}[section]

\newtheorem{lemma}{Lemma}[section]

\renewcommand{\vec}[1]{\boldsymbol{#1}}

\begin{document}

\title{Computation of transmission eigenvalues by the regularized Schur complement for the boundary integral operators }
\date{}

\author{
Yunyun Ma\thanks{
School of Computer Science, Dongguan University of Technology, Dongguan, P. R. China.
{\it mayy007@foxmail.com}}, \  
Fuming Ma\thanks{Institute of Mathematics, Jilin University, Changchun, P. R. China.
{\it mafm@jlu.edu.cn}}, \ 
Yukun Guo\thanks{School of Mathematics, Harbin Institute of Technology, Harbin, P. R. China. {\it ykguo@hit.edu.cn} (Corresponding author)}\ \ and
Jingzhi Li\thanks{Department of Mathematics, Southern University of Science and Technology, Shenzhen, P. R. China. {\it li.jz@sustech.edu.cn}}
}

\maketitle

\begin{abstract}
This paper is devoted to the computation of transmission eigenvalues in the inverse acoustic scattering theory. This problem is first reformulated as a two by two boundary system of boundary integral equations. Next, utilizing the Schur complement technique, we develop a Schur complement operator with regularization to obtain a reduced system of boundary integral equations. The  Nystr\"{o}m discretization is then used to obtain an eigenvalue problem for a matrix. We employ the recursive integral method for the numerical computation of the matrix eigenvalue. Numerical results show that the proposed method is efficient and reduces computational costs. 
\\

\end{abstract}


\section{Introduction}

We consider in this paper the calculation of the transmission eigenvalue problems, which play an important role in scattering theory for inhomogeneous media. Transmission eigenvalues are not only related to the validity of the linear sampling method \cite{Colton1996}, but also give information on the material  properties of the scattering object \cite{Cakoni2016, Colton2019}.  The transmission eigenvalue problem is a boundary value problem for a coupled pair of partial differential  equations in a  bounded domain. But that problem is not covered by the standard theory  of elliptic partial differential equations since it is neither elliptic nor self-adjoint. Hence its study is widely perceived as a challenging issue. Calculating the transmission eigenvalues requires special effort.

Research associated with transmission eigenvalues have focused on two main themes. 
The first concerns are the general theory of  these problems such as the solvability, discreteness and existence, and the spectral properties of the transmission eigenvalue problems \cite{Cakoni2016, Colton2019, Liu2020}.
The mathematical methods for studying these problems include the variational methods and boundary integral equation methods \cite{Cosson2013}. Efficient numerical methods to determine transmission eigenvalues is the second topic. The basis for the numerical techniques for solving the transmission eigenvalues are the finite element \cite{Ji2012, Sun2016},  boundary element methods \cite{Cakoni2017, Cosson2013, Kleefeld2013, Kress2019} and radial basis functions \cite{Kleefeld2018}. Note that the transmission eigenvalue problem is non-linear  and non-selfadjoint. The numerical discretization usually leads to a non-Hermitian and nonlinear matrix eigenvalue problem. That is very challenging in numerical linear algebra. The integral based methods \cite{Beyn2012, Huang2016} were developed to compute the corresponding matrix eigenvalues. An approximation to the eigenvalue  in a given simple closed curve in the complex plane is found by spectral projection using counter integral of the resolvent \cite{Huang2016, Fang2016}.

In this paper, we develop a new integral equation formulation in terms of the Schur complement to a two by two system of boundary integral equations, which is  used to formulate the transmission eigenvalue problem.  
If one of those boundary integral operators is  not invertible, we employ the regularization strategy for modifying the Schur complement. The Nystr\"{o}m  method based on trigonometric interpolation is used to  the discretization of the integral equations for the domains with smooth boundary. For domains with corners, we replace the uniform mesh to sigmoidal-graded meshes. The nonlinear matrix eigenvalue problem is computed by the recursive integral method.  This new method proposed  reduces the computational costs and can be used to study the transmission eigenvalues for a more general refractive index and domain.

We organize this paper in five sections. The boundary integral equation formulations for the transmission eigenvalue problems are developed in Section \ref{sec:preliminary}. We propose a Schur complement with regularization for the  two by two boundary integral equations in Section \ref{sec:schur_complement} and introduce the recursive integral method to find the eigenvalues for a bounded operator. We describe a Nystr\"{o}m discretization for the boundary integral operators and spectral projection  in Section \ref{sec:nystrom}. Finally, numerical results are presented in Section \ref{sec:numerical_example} to confirm the effectiveness of the proposed method. 


\section{Integral equation method}\label{sec:preliminary}

We present in this section the integral equations of the Helmholtz interior transmission problem. The integral equations for that problems are formulated as a $2\times 2$ system of integral equations, which are derived from Green's formulas. 

We begin with a brief introduction to the transmission eigenvalue problem. Let $\Omega\subset\mathbb{R}^2$ be an open bounded Lipschitz domain. Let $\mu$ be the constant refraction index such that $\mu>1$. The transmission eigenvalue problem is to find $\kappa\in\mathbb{C}$ such that there exist non-trivial solutions $w\in L^2(\Omega)$ and $u\in L^2(\Omega)$ with $w-u\in H^2(\Omega)$ satisfying
\begin{align}
\Delta w +\kappa^2 \mu w=0, ~&\text{in}~\Omega,\label{eq:TEP1} \\
\Delta u +\kappa^2 u=0, ~&\text{in}~\Omega, \label{eq:TEP2}\\
w-u=0,~&\text{on}~\Gamma,\label{eq:TEP3}\\
\frac{\partial w}{\partial\vec{\nu}}-\frac{\partial u}{\partial\vec{\nu}}=0,~&\text{on}~\Gamma,\label{eq:TEP4}
\end{align}
where $\Gamma:=\partial\Omega$ and $\vec{\nu}$ is the unit outward normal to $\Gamma$. According to \cite{Colton2019}, any nonzero value $\kappa$ such that there are nontrival solutions $w$ and $u$ of \eqref{eq:TEP1}-\eqref{eq:TEP4} is called a transmission eigenvalue.

We first recall the boundary integral operators. Let $\Phi_{\kappa}$ be the Green's function given by
\begin{equation*}
\Phi_k(x, y):=\dfrac{{\rm i}}{4}H^{(1)}_0(\kappa|x-y|), \quad x, y\in\mathbb{R}^2,
\end{equation*}
where $H_0^{(1)}$ is the Hankel function of the first kind of order zero. The single and double layer potentials are
defined by
\begin{equation*}\label{Sec2:singleLayer_P}
(\mathcal{SL}_{\kappa}[\phi])(x):=\int_{\Gamma}\Phi_k(x,y)\phi(y){\rm d}s(y),\quad x\in\mathbb{R}^2\backslash\Gamma,
\end{equation*}
and 
\begin{equation*}\label{Sec2:doubleLayer_P}
(\mathcal{DL}_{\kappa}[\phi])(x):=\int_{\Gamma}\dfrac{\partial\Phi_k(x, y)}{\partial \vec{\nu}(y)}\phi(y){\rm d}s(y),\quad x\in\mathbb{R}^2\backslash\Gamma,
\end{equation*}
where $\phi$ is an integrable function. The interior Dirichlet traces on $\Gamma$ of the single and double layer potentials are given by
\begin{align}
(\mathcal{SL}_{\kappa}[\phi])^{-} = & \mathcal{S}_{\kappa}[\phi] \label{Sec2:signleLayer_DJump} \\
(\mathcal{DL}_{\kappa}[\phi])^{-} = & \mathcal{D}_{\kappa}[\phi]-\dfrac{\phi}{2}, \label{Sec2:doubleLayer_DJump}
\end{align}
where the single and double layer operators are defined by
\begin{equation*}\label{Sec2:singleLayer_O}
(\mathcal{S}_{\kappa}[\phi])(x):=\int_{\Gamma}\Phi_k(x,y)\phi(y){\rm d}s(y),\quad x\in\Gamma,
\end{equation*}
and 
\begin{equation}\label{Sec2:doubleLayer_O}
(\mathcal{D}_{\kappa}[\phi])(x):=\int_{\Gamma}\dfrac{\partial\Phi_k(x, y)}{\partial \vec{\nu}(y)}\phi(y){\rm d}s(y),\quad x\in\Gamma.
\end{equation}

We now present the integral formulations for the transmission eigenvalue problem. To this end, we denote that $\kappa_1:=\kappa\sqrt{\mu}$ and
\begin{equation*}
\alpha:=\dfrac{\partial w }{\partial\vec\nu},\quad\beta:=w,~\text {on}~\Gamma.
\end{equation*}
According to the boundary conditions \eqref{eq:TEP3}-\eqref{eq:TEP4}, we have that 
\begin{equation*}
\dfrac{\partial u}{\partial\nu}=\alpha~\text{and}~u=\beta,~\text {on}~\Gamma.
\end{equation*}
We then have the following integral representation 
\begin{equation}\label{Sec2:integralRepresentation_wu}
w=\mathcal{SL}_{\kappa_1}[\alpha]-\mathcal{DL}_{\kappa_1}[\beta]~\text{and}~
u=\mathcal{SL}_{\kappa}[\alpha]-\mathcal{DL}_{\kappa}[\beta]~\text{in} ~\Omega.
\end{equation}
This together with \eqref{Sec2:signleLayer_DJump} and \eqref{Sec2:doubleLayer_DJump} yields that
\begin{equation}\label{Sec2:DirichletTraces} 
\beta=\mathcal{S}_{\kappa_1}[\alpha]-\mathcal{D}_{\kappa_1}[\beta]+\dfrac{\beta}{2}~\text{and}~
\beta=\mathcal{S}_{\kappa}[\alpha]-\mathcal{D}_{\kappa}[\beta]+\dfrac{\beta}{2}~\text{on} ~\Gamma.
\end{equation}
From \eqref{Sec2:DirichletTraces}, we obtain the following $2\times 2$ system of  boundary  integral equations:
\begin{equation}\label{Sec2:IE_1}
\mathcal{Z}(\kappa)
\begin{pmatrix}
\beta\\
\alpha
\end{pmatrix}=
\begin{pmatrix}
0 \\
0
\end{pmatrix},
\end{equation}
where
\begin{equation*}
\mathcal{Z}(\kappa)=
\begin{pmatrix}
\mathcal{I}/2+\mathcal{D}_{\kappa}& -\mathcal{S}_{\kappa} \\
\mathcal{I}/2+\mathcal{D}_{\kappa_1}&- \mathcal{S}_{\kappa_1} 
\end{pmatrix}.
\end{equation*}
and $\mathcal{I}$ denotes the identity operator.
The transmission eigenvalues are $\kappa$'s satisfying \eqref{Sec2:IE_1}.
This implies that $\kappa$ is a transmission eigenvalue if zero is an eigenvalue of $\mathcal{Z}$.

We now recall some properties of the aforementioned integral operators.

\begin{lemma}[\cite{Kirsch1989, Mclean2000}]\label{lem:1}
Let $\Gamma$ be of class $C^{2,1}$. Then we have
 \begin{enumerate}
 \item[(i)] The operator $\mathcal{S}_{\kappa}: H^{1/2}(\Gamma)\to H^{3/2} (\Gamma)$ is Fredholm with index zero.
 \item[(ii)] The operator $\mathcal{D}_{\kappa}: H^{3/2}(\Gamma)\to H^{3/2} (\Gamma)$ is compact.
 \end{enumerate}
\end{lemma}

The following theorem follows directly from Lemma \ref{lem:1}.
\begin{theorem}\label{thm:1}
Let $\Gamma$ be of class $C^{2,1}$. Then the operator  
$\mathcal{Z}(\kappa): H^{3/2}(\Gamma)\times H^{1/2}(\Gamma)\to H^{3/2}(\Gamma)\times H^{3/2}(\Gamma)$ is Fredholm of index zero and analytic on $\kappa\in\mathbb{C}\backslash\mathbb{R}^-$.
\end{theorem}

\section{The Schur complement and recursive integral method}\label{sec:schur_complement}

We present in this section the Schur complement with regularization for the system of the boundary integral equations \eqref{Sec2:IE_1}. This leads to a nonlinear  and non-selfadjoint eigenvalue problem.  We use contour integral based on spectral projection to test if zero is an eigenvalue of the corresponding operators.

We first introduce Schur complement  for the  $2\times2$ block operator $\mathcal{Z}$. 
If zero is an eigenvalue of $\mathcal{Z}(\kappa)$ for $\kappa\in\mathbb{C}\backslash\mathbb{R}^-$, we have nontrivial solutions to the equations
\begin{align}
(\mathcal{I}/2+\mathcal{D}_{\kappa}) \beta-\mathcal{S}_{\kappa}\alpha=0,\label{eq:eq1}\\
(\mathcal{I}/2+\mathcal{D}_{\kappa_1})\beta-\mathcal{S}_{\kappa_1}\alpha=0.\label{eq:eq2}
\end{align}
If we assume that $\mathcal{S}_{\kappa_1}$ is invertible, we first
solve for $\alpha$ from \eqref{eq:eq2}, getting
\begin{equation*} 
\alpha=\mathcal{S}_{\kappa_1}^{-1}(\mathcal{I}/2+\mathcal{D}_{\kappa_1})\beta.
\end{equation*}
and  substituting this expression for $\alpha$ in the  equation \eqref{eq:eq1}, we obtain that 
\begin{equation*}
[(\mathcal{I}/2+\mathcal{D}_{\kappa})-\mathcal{S}_{\kappa}\mathcal{S}_{\kappa_1}^{-1}(\mathcal{I}/2+\mathcal{D}_{\kappa_1})]\beta=0.
\end{equation*}
If $\mathcal{S}_{\kappa_1}$  is invertible,  we call the operator
\begin{equation}\label{Shur1}
\mathcal{A}(\kappa):=(\mathcal{I}/2+\mathcal{D}_{\kappa})-\mathcal{S}_{\kappa} \mathcal{S}_{\kappa_1}^{-1} (\mathcal{I}/2+\mathcal{D}_{\kappa_1})
\end{equation}
 the Schur complement of  $\mathcal{S}_{\kappa_1} $ in $\mathcal{Z}_1(\kappa)$. 
We conclude that if zero is an eigenvalue of $\mathcal{A}(\kappa)$,
$\mathcal{Z}(\kappa)$ has an eigenvalue equal to zero. The transmission eigenvalues are $\kappa$'s satisfying $\mathcal{A}(\kappa)\beta=0$.

We remark in passing that the single  layer operator $\mathcal{S}_{\kappa_1}$ is invertible for $\Omega$ with $C^{2,1}$ boundary, if $\kappa_1^2$ is not a Dirichlet eigenvalue for $-\triangle$ in $\Omega$.
For  the above approach, we have to exclude the Dirichlet eigenvalue in $\Omega$.
To overcome this weakness, we propose schur complement with a regularization for $\mathcal{Z}$. 

We now develop a regularization technique for solving \eqref{eq:eq1}-\eqref{eq:eq2} for the case 
$\mathcal{S}_{\kappa_1}$ is not invertible.
Let $\eta>0$ be a regularization parameter.  Following the idea of the standard Tikhonov regularization, we introduce the operator
\begin{equation}\label{Shur1_Regularization}
\mathcal{A}^{\eta}(\kappa):=(\mathcal{I}/2+\mathcal{D}_{\kappa})-\mathcal{S}_{\kappa} (\eta\mathcal{I}+\mathcal{S}_{\kappa_1}^*\mathcal{S}_{\kappa_1})^{-1}\mathcal{S}_{\kappa_1}^* (\mathcal{I}/2+\mathcal{D}_{\kappa_1}),
\end{equation}
which is called the regularized  Schur complement of $\mathcal{S}_{\kappa_1}$
 in $\mathcal{Z}_1(\kappa)$. We note that $\mathcal{A}^{\eta}=\mathcal{A}$ for $\eta=0$.
According to Lemma \ref{lem:1}, we have the following Fredholm property of $\mathcal{A}^{\eta}$.

\begin{theorem}
Let $\Gamma$ be of class $C^{2,1}$ and $\eta\geq 0$. Then the operator  
$\mathcal{A}^\eta(\kappa): H^{3/2}(\Gamma)\to H^{3/2}(\Gamma)$ is Fredholm of index zero and analytic on $\kappa\in\mathbb{C}\backslash\mathbb{R}^-$.
\end{theorem}

In the following part, we present the method to test if zero is an eigenvalue of $A^{\eta}(\kappa)$. 
Note that $\mathcal{A}^\eta(\kappa)\beta=0$ is a  nonlinear integral eigenvalue problem and the 
discretization for the integral equations by the Nystr\"{o}m  method leads to a dense matrices.
We shall use  the recursive integral method to test  that  zero is an eigenvalue of $\mathcal{A}^{\eta}(\kappa)$ or not, and recall that method as follows.
We define the resolvent set of $\mathcal{A}^{\eta}(\kappa)$ by
\begin{equation*}
\rho(\mathcal{A}^{\eta}(\kappa)):=\{z\in\mathbb{C}:(z\mathcal{I}-\mathcal{A}^{\eta}(\kappa))^{-1} \in\mathit{B}(\mathbb{X}) \},
\end{equation*}
and its spectrum $\sigma(\mathcal{A}^{\eta}(\kappa))$, where $\mathbb{X}$ is a Banach space. Let the spectral projection $\mathcal{P}$ associated with  $\mathcal{A}^{\eta}(\kappa)$ and zero denote by
\begin{equation}\label{eq:proj}
\mathcal{P}:=\dfrac{1}{2\pi{\rm i}}\int_{\gamma}(z\mathcal{I}-\mathcal{A}^{\eta}(\kappa))^{-1}{\rm d}z,
\end{equation}
where $\gamma$ is a closed rectifiable curve on the complex plane in $\rho(\mathcal{A}^{\eta}(\kappa))$ enclosing zero, but no any other point in  $\sigma(\mathcal{A}^{\eta}(\kappa))$. If there are no eigenvalues inside $\gamma$, we  have that $\mathcal{P}f=0$ for $f\in\mathbb{X}$. Let $f$ be randomly chosen and $\gamma$ be a circle with small diameter. $\|\mathcal{P}^2f\|$ is used to decide whether zero is an eigenvalue of $\mathcal{A}^{\eta}(\kappa)$ or not, since $\mathcal{P}^2=\mathcal{P}$. Hence, we need not to compute the eigenvalues of  $\mathcal{A}^{\eta}(\kappa)$,  and compare them with zero.

\section{The Nystr\"{o}m  discretization}\label{sec:nystrom}

We present in this section the Nystr\"{o}m  discretization of  the operator $\mathcal{Z}(\kappa)$  and the spectral projection $\mathcal{P}$ for completeness. We refer the readers to  \cite{Kress1995, Kress2014} for more details on the Nystr\"{o}m  discretization for domains with smooth boundary, and \cite{Dom2016} for Lipschitz domains.

We first parametrize the boundary integral operators $\mathcal{S}_\kappa$ and $\mathcal{D}_\kappa $. We assume that the boundary curve $\Gamma$ is described by a $2\pi$-periodic parametric representation of the form
\begin{equation*}
\Gamma:=\{z(t):=(z_1(t),z_2(t))^\top: t\in[0,2\pi]\}.
\end{equation*}  
Let $J_n$ denote the Bessel function of the first kind of order $n$. The parameterized operator $\mathcal{S}_\kappa$ is
\begin{equation*}
(\mathcal{S}_\kappa[\phi])(s)=\int_0^{2\pi}K^S(s,t;\kappa)\varphi(t)\,{\rm d}t,~\text{for}~s\in[0,2\pi],
\end{equation*}
where $\varphi(t)=\phi(z(t))$ and
\begin{equation*}
K^S(s, t;\kappa)=\dfrac{\rm i}{4}H_0^{(1)}(\kappa|z(s)-z(t)|)=K^S_1(s, t;\kappa)\ln\left\{4\sin^2{\dfrac{s-t}{2}}\right\}+K^S_2(s, t;\kappa) 
\end{equation*}
with
\begin{align*}
K^S_1(s, t;\kappa)= & -\dfrac{1}{4\pi}J_0(\kappa|z(s)-z(t)|)|z'(t)|,\\
K^S_2(s, t;\kappa)= & K^S(s, t;\kappa)-K^S_1(s, t;\kappa)\ln\left\{4\sin^2\dfrac{s-t}{2}\right\}.
\end{align*}
Note that 
\begin{equation*}
K^S_1(t,t;\kappa)=-\dfrac{1}{4\pi}|z'(t)|~\text{and}~K^S_2(t,t;\kappa)=\left(\dfrac{\rm i}{4}-\dfrac{\gamma}{2\pi}-\dfrac{1}{4\pi}\left(\dfrac{\kappa^2}{4}|z'(t)|^2\right) \right)|z'(t)|,
\end{equation*}
with Euler's constant $\gamma$. 

Recalling the definition \eqref{Sec2:doubleLayer_O}, we notice that
the kernel of  $\mathcal{D}_\kappa$ has singularity at the corners and the definition of that integral is understood in the sense of Cauchy principal value integral. We split the kernel into smooth and singular components. The parameterized operator $\mathcal{D}_\kappa$ is
\begin{align*}
(\mathcal{D}_\kappa[\phi])(s)=&\int_0^{2\pi}(K^D(s, t;\kappa)-K^D(s, t;0))\varphi(t){\rm d}t\\
&+\int_0^{2\pi}K^D(s,t;0)(\varphi(t)-\varphi(s)){\rm d}t\\
&+\varphi(s)\int_0^{2\pi}K^D(s, t;0){\rm d}t, 
~\text{for}~s\in[0,2\pi],
\end{align*}
where
\begin{align*}
K^D(s, t;\kappa) = & \dfrac{\rm i\kappa}{4}\dfrac{(z(s)-z(t))\cdot\vec{\nu}(z(t))}{|z(s)-z(t)|}H_1^{(1)}(\kappa|z(s)-z(t)|)|z'(t)|\\
= & K^D_1(s, t)\ln\left\{4\sin^2\dfrac{s-t}{2}\right\}+K^D_2(s,t), 
\end{align*}
with
\begin{align*}
K^D_1(s, t;\kappa)= & -\dfrac{\kappa}{4\pi}\dfrac{(z(s)-z(t))\cdot\vec{\nu}(z(t))}{|z(s)-z(t)|}J_1(\kappa|z(s)-z(t)|)|z'(t)|,\\
K^D_2(s, t;\kappa)= & K^D(s, t;\kappa)-K^D_1(s, t;\kappa)\ln\left\{4\sin^2\dfrac{s-t}{2}\right\},
\end{align*}
and 
\begin{align*}
K^D(s, t;0)= & \dfrac{1}{2\pi}\dfrac{(z(s)-z(t))\cdot\vec{\nu}(z(t))}{|z(s)-z(t)|^2}|z'(t)|
\end{align*}
Note that 
\begin{equation*}
K^D_2(t, t;\kappa)=K^D(t, t;0)=\dfrac{1}{4\pi}\dfrac{z''(t)\cdot\vec{\nu}(z(t))}{|z'(t)|}.
\end{equation*}

We now approximate the integral operators $\mathcal{Z}(\kappa)$ as follows. For the $2\pi$-periodic integrands we choose an equidistant set of knots $t_j:=\pi j/n$ with $n\in\mathbb{N}$ and $j=0,1,\ldots,2n-1$. We approximate the operators $\mathcal{S}_\kappa$ and $\mathcal{D}_\kappa$ by 
 the quadratures
\begin{equation}\label{Rule1}
\int_{0}^{2\pi}\varphi(t)\,{\rm d}t\approx\dfrac{\pi}{n}\sum\limits_{j=0}^{2n-1}\varphi(t_j),
\end{equation}
\begin{equation}\label{Rule2}
\int_{0}^{2\pi}\varphi(t)\ln\left\{\{4\sin^2\dfrac{s-t}{2}\right\}\, {\rm d}t\approx\sum\limits_{j=0}^{2n-1}R_j^{(n)}(s)\varphi(t_j),
\end{equation}
with
\begin{equation*}
R_j^{(n)}(s)=\dfrac{2\pi}{n}\sum_{\ell=1}^{n-1}\dfrac{1}{\ell}\cos{(\ell(s-t_j))}-\dfrac{\pi}{n^2}\cos{(n(s-t_j))}.
\end{equation*}
The sequences of numerical integration operators for $\mathcal{S}_{\kappa_1}$, $\mathcal{S}_{\kappa}$, $\mathcal{D}_{\kappa_1}$ and $\mathcal{D}_{\kappa}$
are denoted by 
$\mathcal{S}_{\kappa_1,n}$, $\mathcal{S}_{\kappa,n}$, 
$\mathcal{D}_{\kappa_1,n}$ and $\mathcal{D}_{\kappa,n}$.
We then approximate $\mathcal{Z}(\kappa)$  by a sequence of numerical integration operators
\begin{equation*}
\mathcal{Z}_{n}(\kappa)=
\begin{pmatrix}
\mathcal{I}_n/2+\mathcal{D}_{\kappa,n}& -\mathcal{S}_{\kappa,n} \\
\mathcal{I}_n/2+\mathcal{D}_{\kappa_1,n}&-\mathcal{S}_{\kappa_1,n} 
\end{pmatrix}.
\end{equation*}
The transmission eigenvalues are  approximated by $\kappa$'s satisfying 
\begin{equation}\label{Sec4}
\mathcal{Z}_{n}(\kappa)
\begin{pmatrix}
\beta\\
\alpha
\end{pmatrix}=
\begin{pmatrix}
0 \\
0
\end{pmatrix}.
\end{equation}

In passing, we comment on that the eigenpairs $\mathcal{Z}_{n}$  and $\mathcal{Z}$ are related to each other  in some sense.  We denote the space of  trigonometric polynomials by  $$\mathbb{T}_n:={\rm span}\{1,\cos{nt},\cos{mt},\sin{mt}:m=1,2,\ldots,n-1\}.$$ For $\mathbf{g}_n:=[g_0,g_1,\ldots,g_{2n-1}]^\top\in\mathbb{C}^{2n}$, there exists a unique  trigonometric polynomial $v_n\in\mathbb{T}_n$ of the form
$$
v_n(t):=\sum_{j=0}^{n}a_j\cos{jt}+\sum_{j=1}^{n-1}b_j\sin{jt},
$$
satisfying $v_n(t_j)=g_j$ ($j=0$, $1$, $\ldots$, $2n-1$) and $v_n\in {\rm C}[0,2\pi]$. This implies that  $ \mathcal{Z}_{n} $  and $\mathcal{Z}$ are bounded operators on $[{\rm C}[0,2\pi]]^2$.
Applying the pointwise convergence of the Nystr\"{o}m method \cite{Kress1989}, we obtain
 the spectral properties of  $ \mathcal{Z}_{n} $  and $\mathcal{Z}$ in $[{\rm C}[0,2\pi]]^2$ 
in the following theorem.

\begin{theorem}\label{thm:sec4}
If $\kappa\in\mathbb{C}\backslash\mathbb{R}^-$ is not an eigenvalue of $\mathcal{Z}$, then there exists a positive integral $n_0$, such that for all $n\in\mathbb{N}$ with $n>n_0$,
$\kappa$ is also not an eigenvalue of $\mathcal{Z}_{n}$.
\end{theorem}

\begin{proof}  We assume that  $\kappa\in\mathbb{C}\backslash\mathbb{R}^-$ is not an eigenvalue of $\mathcal{Z}$. This implies that $\mathcal{Z}^{-1}$ is bounded. For any $v \in [{\rm C}[0,2\pi]]^2$ with $\|v\|_\infty=1$, we define $\xi:=v/\|\mathcal{Z}v\|_\infty$.  Note that $\|\xi\|_\infty\leq\|\mathcal{Z}^{-1}\|$. There exists a positive integer $n_0$ such that for all $n\in\mathbb{N}$ with $n>n_0$,
\begin{align*}
	\|\mathcal{Z}\xi-\mathcal{Z}^{-1} \mathcal{Z}_{n}\mathcal{Z}\xi\|_\infty =&
	\|\mathcal{Z}^{-1}(\mathcal{Z}- \mathcal{Z}_{n})\mathcal{Z}\xi\|_\infty \\
	\leq & \|\mathcal{Z}^{-1}\|\|(\mathcal{Z}- \mathcal{Z}_{n})\mathcal{Z}\|\|\xi\|_\infty \\
	\leq &\|\mathcal{Z}^{-1}\|^2\|(\mathcal{Z}- \mathcal{Z}_{n})\mathcal{Z}\| \\
	< & \frac12.
\end{align*}
For any $v_n\in\mathbb{T}_n$ with $\|v_n\|_\infty=1$, let $\xi_n:=v_n/\|\mathcal{Z}v_n\|_\infty$. This leads to that $\|\mathcal{Z}\xi_n\|_\infty=1$. We obtain that there 
exists a positive integer $n_0$ such that for all $n\in\mathbb{N}$ with $n>n_0$
\begin{equation*}
\| \mathcal{Z}^{-1} \mathcal{Z}_{n}\mathcal{Z}\xi_n\|_\infty\geq \|\mathcal{Z}\xi_n\|_\infty- \|\mathcal{Z}\xi_n- \mathcal{Z}^{-1} \mathcal{Z}_{n}\mathcal{Z}\xi_n\|_\infty\geq \frac12.
\end{equation*}
This yields that $\kappa$ is  not an eigenvalue of $\mathcal{Z}_{n}$ and completes the proof.
\end{proof}

We conclude  from Theorem \ref{thm:sec4}  that 
$\mathcal{Z}_{n}$ is free of spurious eigenvalues  of $\mathcal{Z}$ for $n$ large enough. 
We determine the transmission eigenvalues by \eqref{Sec4}. 
The Schur complement of $\mathcal{S}_{\kappa_1,n}$
 in $\mathcal{Z}_{n}$ with regularization is given by
\begin{equation*}
\mathcal{A}_{n}^{\eta}(\kappa):=(\mathcal{I}/2+\mathcal{D}_{\kappa,n})-\mathcal{S}_{\kappa,n} (\eta\mathcal{I}+\mathcal{S}_{\kappa_1,n}^*\mathcal{S}_{\kappa_1,n})^{-1}\mathcal{S}_{\kappa_1,n}^* (\mathcal{I}/2+\mathcal{D}_{\kappa_1,n}),
\end{equation*}
where $\eta>0$.
For the case $\mathcal{S}_{\kappa_1,n}$  is invertible, the Schur complement of  $\mathcal{S}_{\kappa_1,n} $ in $\mathcal{Z}_{n}$ is given by
\begin{equation*}
\mathcal{A}_{n}^0(\kappa):=(\mathcal{I}/2+\mathcal{D}_{\kappa,n})-\mathcal{S}_{\kappa,n} \mathcal{S}_{\kappa_1,n}^{-1} (\mathcal{I}/2+\mathcal{D}_{\kappa_1,n}).
\end{equation*}

We then consider the Nystr\"{o}m  discretization for the domain with corners. 
We assume that the domain $\Omega$ has corners at $z(T_j)$ for $j=1,2,\ldots m$, where
$0\leq T_1<T_2<\cdots<T_m<2\pi$.  $z_1$ and $z_2$ are smooth with $[z_1'(t)]^2+[z_2'(t)]^2>0$ 
on each interval  $t\in[T_j,T_{j+1}]$ for $j=1,2,\ldots,m-1$. We replace the equidistant mesh by 
a sigmoidal-graded mesh through substituting a new variable based on the sigmoid transform \cite{Cosson2013, Kress1990}.
Let $T_0=0$ and $T_{m+1}=2\pi$.  For $j=0,1,2,\ldots,m$,  we define the sigmoid transform by
\begin{equation*}
w(s):=\dfrac{T_{j+1}[v(s)]^p+T_j[1-v(s)]^p}{[v(s)]^p+[1-v(s)]^p},~s\in [T_j,T_{j+1}],
\end{equation*}
\begin{equation*}
v(s)=\left(\dfrac{1}{p}-\dfrac{1}{2}\right)\left(\dfrac{T_j+T_{j+1}-2s}{T_{j+1}-T_j}\right)^3+
\dfrac{1}{2}\dfrac{2s-T_j-T_{j+1}}{T_{j+1}-T_j}+\dfrac{1}{2},
\end{equation*}
where $p\geq2$. By a change of variables $t\mapsto w(t)$, we obtain the new parametrization
\begin{equation*}
\Gamma:=\{z(w(t)):=(z_1(w(t)),z_2(w(t)))^\top: t\in[0,2\pi]\}.
\end{equation*}  
With substitution of the new parametrization above, we obtain the sequences of numerical integration operators for $\mathcal{S}_{\kappa_1}$, $\mathcal{S}_{\kappa}$, $\mathcal{D}_{\kappa_1}$ and $\mathcal{D}_{\kappa}$ based on the equidistant mesh, which like the domain with smooth boundary.

We next approximate the spectral projection $\mathcal{P}$ defined by \eqref{eq:proj}.
We choose $\gamma$ as a small circle centered  at the origin with radius $r\in(0,1)$. We substitute
\begin{equation*}
\gamma:=\{z=r{\rm e}^{{\rm i }\theta}:~\theta\in[0,2\pi]\}
\end{equation*}
to obtain that 
\begin{equation*}
\mathcal{P}=\dfrac{1}{2\pi}\int_{0}^{2\pi}r{\rm e}^{{\rm i }\theta}(r{\rm e}^{{\rm i }\theta}\mathcal{I}-\mathcal{A}^{\eta}(\kappa))^{-1}{\rm d}\theta.
\end{equation*}
The  quadrature rule \eqref{Rule1} is employed for the numerical computation of  $\mathcal{P}$.
For fixed $m\in\mathbb{N}$, and $\vec{f}\in\mathbb{C}^{2n}$, the approximation $\mathcal{P}_m$ is computed by the quadrature rule
\begin{equation}
\mathcal{P}_m\vec{f}:=\dfrac{1}{2m}\sum_{j=0}^{2m-1}r{\rm e}^{{\rm i }\theta_j}\vec{x}_j,
\end{equation}
where $\theta_j:=\pi j/m$ for $j=0$, $1$, $\ldots$, $2m-1$ are quadrature points, and $\vec{x_j}$ are the solutions of the following linear system
\begin{equation*}
(r{\rm e}^{{\rm i }\theta_j}\mathcal{I}-\mathcal{A}_n^{\eta}(\kappa))\vec{x}_j=\vec{f}.
\end{equation*} 
 We finally present the algorithm for testing whether zero is an eigenvalue of  $A_{n}^{\eta}$ or not.
 \begin{algorithm}
 \renewcommand{\thealgorithm}{}
 \caption{Numerical computation of the transmission eigenvalues} 
\begin{enumerate}
 \item[Step 1] Choose $\kappa\in (a,b)$, which is  an interval of wavenumbers.
 \item[Step 2] For fixed $n\in\mathbb{N}$, calculate $\mathcal{S}_{\kappa_1,n}$, $\mathcal{S}_{\kappa,n}$, $\mathcal{D}_{\kappa_1, n}$ and  $\mathcal{D}_{\kappa, n}$.
\item [Step 3] Choose $\eta\in[0,1)$ to obtain  $\mathcal{A}_{n}^{\eta}(\kappa)$, where the regularization parameter $\eta>0$ for the case that  the condition number  of  $\mathcal{S}_{\kappa_1,n}$ is large.
\item [Step 4] For $m\in\mathbb{N}$, $r\ll 1$ and a random $\vec{f}\in\mathbb{C}^{2n}$, compute
\begin{equation*}
{\rm  RIM}_m(\kappa):=\left\|\mathcal{P}_m\left[\dfrac{\mathcal{P}_mf}{\|\mathcal{P}_mf\|}\right]\right\|.
\end{equation*}
\item [Step 5] Decide if $\gamma$ contains an eigenvalue  and $\kappa$ is a transmission eigenvalue.
 \end{enumerate}
\end{algorithm}


\section{Numerical results}\label{sec:numerical_example}

We shall illustrate in this section the computation of interior transmission eigenvalues by using \eqref{Shur1}. The index of refraction is chosen as $\mu=16$.

We start with an interval $(a, b)$ of wavenumbers and uniformly divide it into $N$ subintervals. For each wavenumber, the boundary integral operators are discretized with $n=32$, namely, 64 quadrature nodes over $[0, 2\pi]$. We set $m=64$.

\begin{example}
Let $\Omega$ be a disk with radius $1/2$. In this case, the exact transmission eigenvalues are $\kappa$'s such that \cite{colton2010analytical, Fang2016}
\begin{equation*}
J_1(\kappa/2)J_0(2\kappa)-4J_0(\kappa/2)J_1(2\kappa)=0,
\end{equation*}
and 
\begin{equation*}
J_{m-1}(\kappa/2)J_m(2\kappa)-4J_m(\kappa/2)J_{m-1}(2\kappa)=0,
\end{equation*}
for $m\in\mathbb{N}$. The exact values in $[1.5, 5]$ are given by
\begin{equation*}
\kappa_1=1.9880,~\kappa_2=2.6129,~\kappa_3=3.2267,~\kappa_4=3.7409,~\kappa_5=3.8264,
\end{equation*}
\begin{equation*}
\kappa_6=4.2958,~\kappa_7=4.4154,~\kappa_8=4.9418,~\kappa_9=4.9959.
\end{equation*}
\end{example}

The numerical results are presented in Figures \ref{fig:disk_r}--\ref{fig:disk_complex}. We mark each location of the  exact eigenvalues by  a red line. We first  choose the interval to be $[1.6, 2.2]$  with dividing it into $100$ subintervals, where the value of ${\rm  RIM}_m(\kappa)$ is plotted in Figure \ref{fig:disk_r} with radii $r=0.005, 0.003$ and $0.001$. We see that the result for the circle with radius $r=0.001$ is better than the other two cases. This implies that the effectiveness of recursive integral method is affected by the radius of the circle. We shall choose $r<0.05$ in the following examples. We choose the interval to be  $[2.3, 2.8]$ and $[3, 3.5]$ with dividing it into $100$ subintervals in Figure \ref{fig:disk_interval}. We choose the interval to be  $[3.5, 4]$ with dividing it into 100 subintervals, and $[4, 5]$ with dividing it into 200 subintervals in Figure \ref{fig:disk_interval}.
 We see that the value of ${\rm  RIM}_m(\kappa)$ is zero except the location near the eigenvalues.
 This  concurs with the theoretical estimate. We finally search for transmission eigenvalues in the complex plane $\mathbb{C}$. The real part is seted in the interval $[4.85, 4.95]$ and divided it into $20$ subinterval. The imaginary part  is chosen in the intervals $[0.5,0.7]$ and $[-0.7,-0.5]$, where each interval is divided into $200$ subintervals. The result are presented in Figure \ref{fig:disk_complex}. We find that there exist a pair of complex eigenvalues around $k=4.90\pm0.58{\rm i}$.

\begin{figure}
 \centering
  \subfigure[]{\includegraphics[clip, trim = 0 220 0 230, width=.48\textwidth]{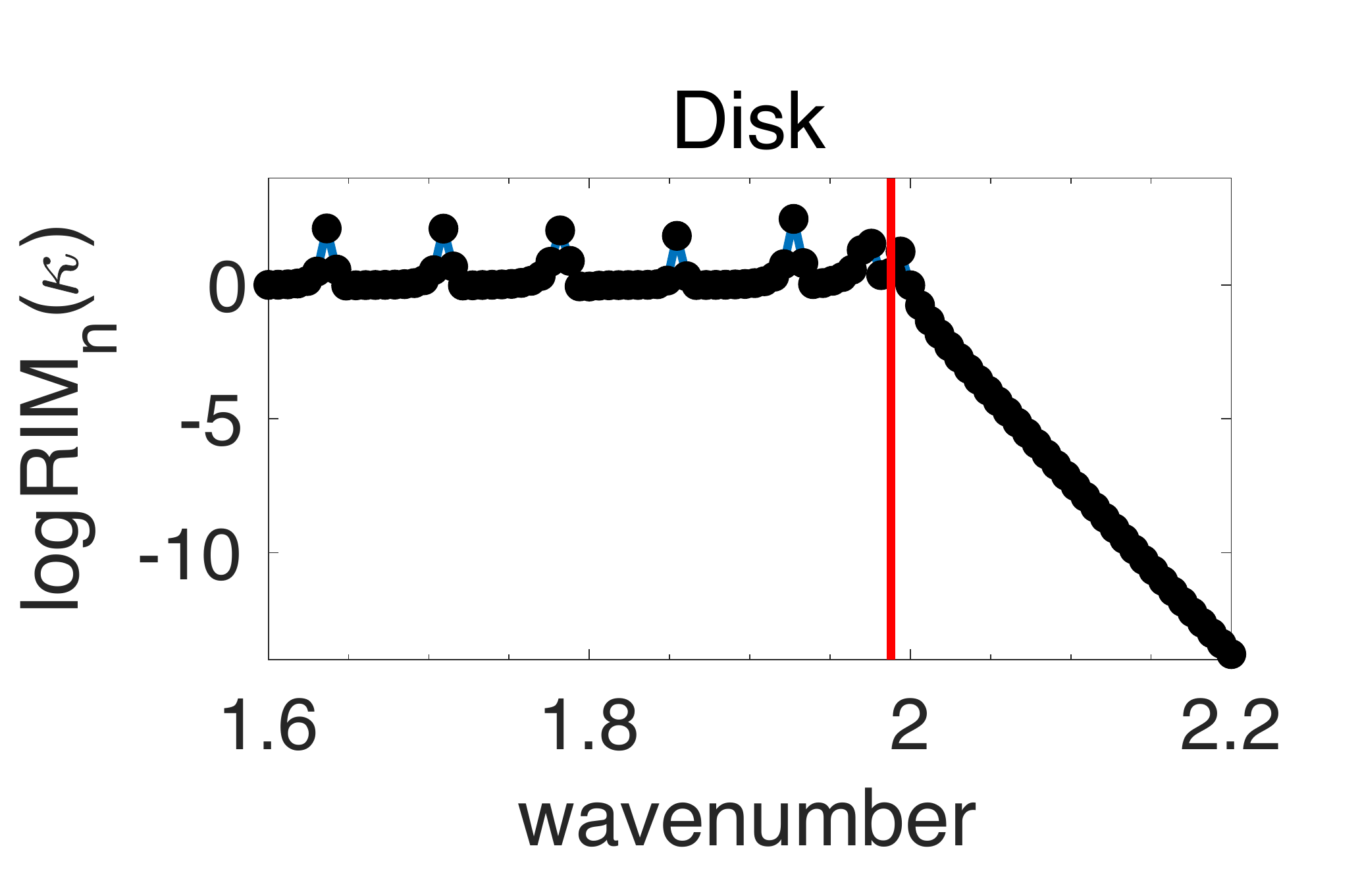}}
  \subfigure[]{\includegraphics[clip, trim = 0 220 0 230, width=.48\textwidth]{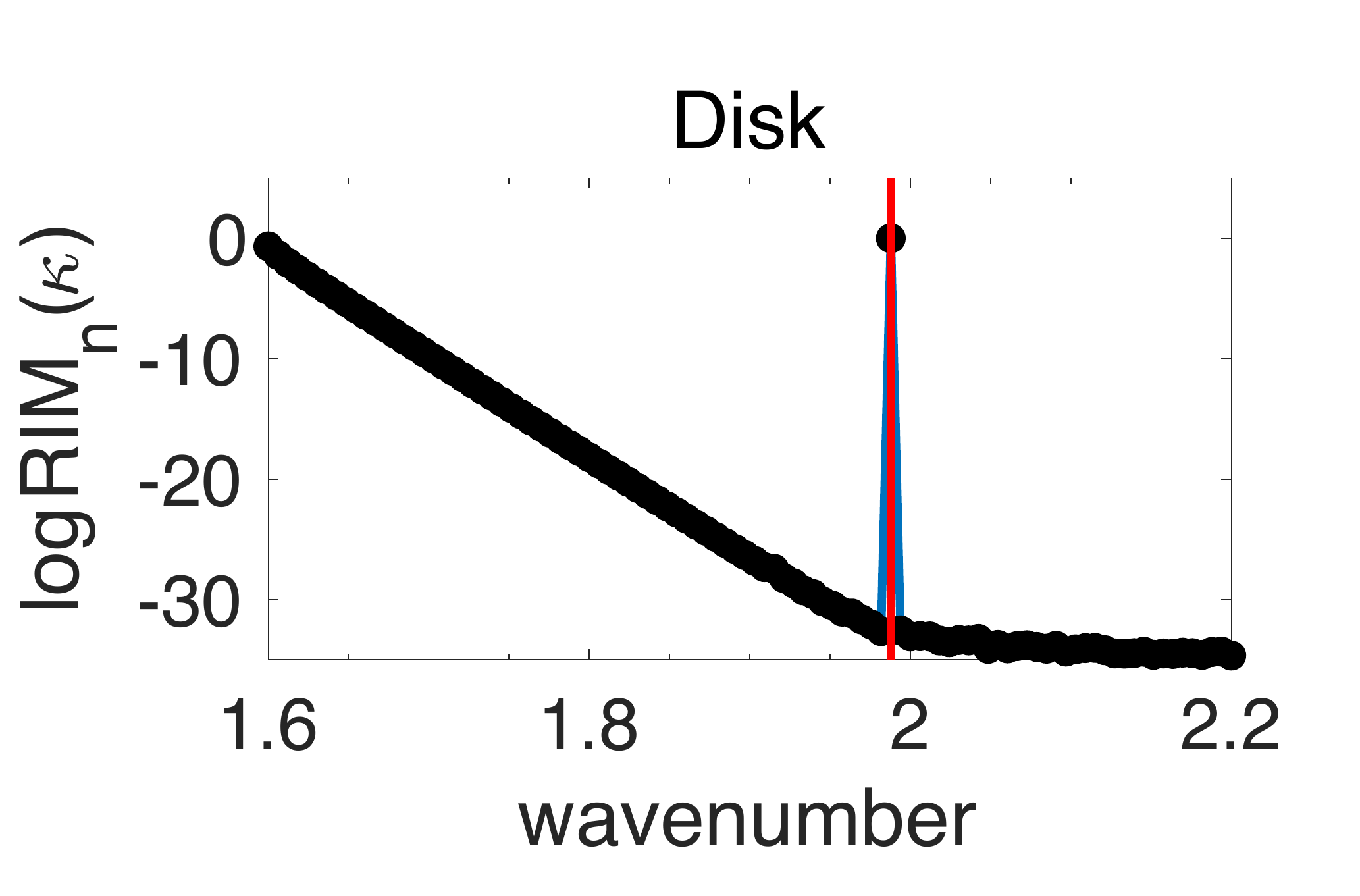}}\\
  \subfigure[]{\includegraphics[clip, trim = 0 220 0 230, width=.48\textwidth]{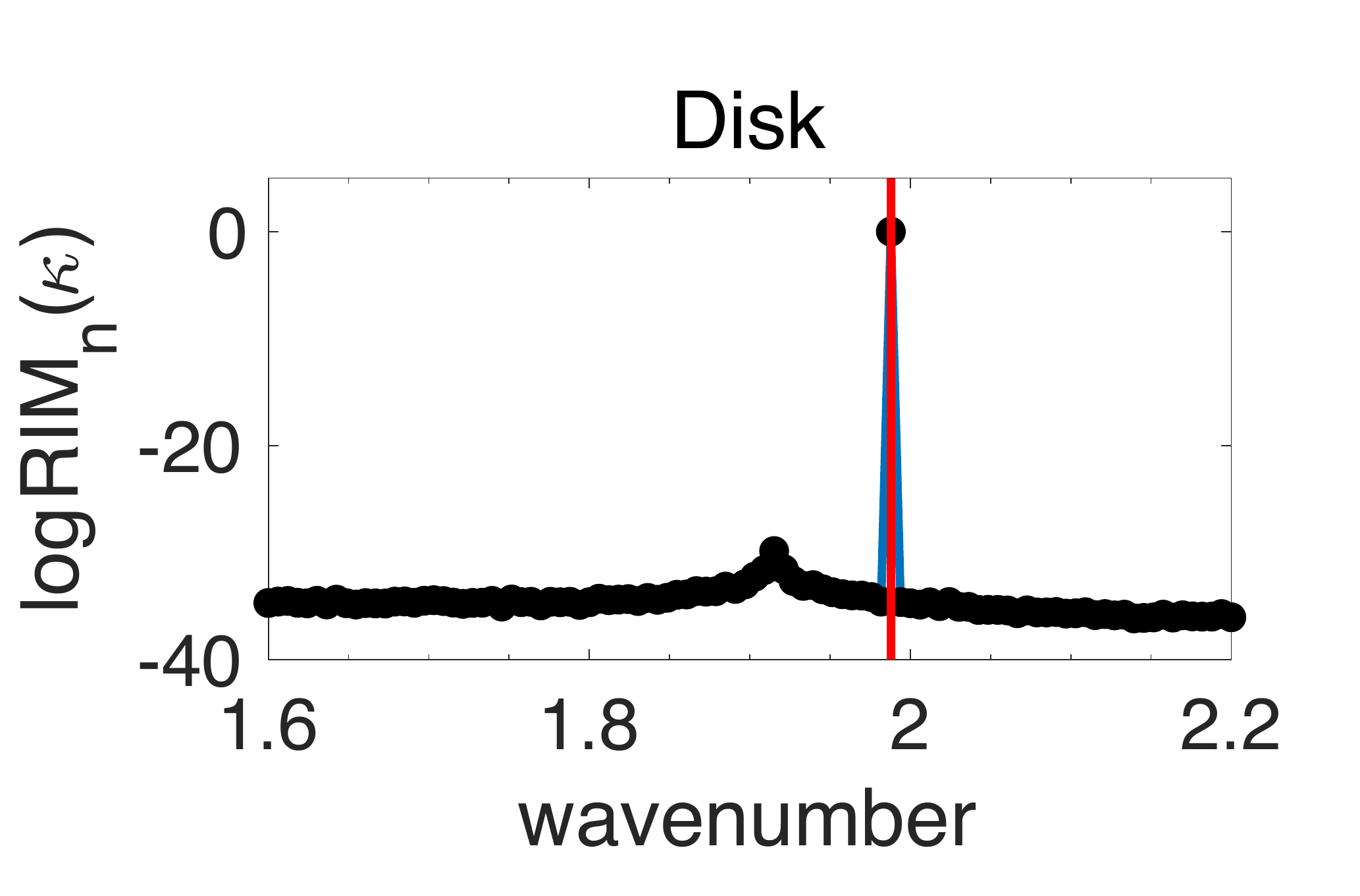}}
  \caption{Plots of $\log{{\rm  RIM}_n(\kappa)}$ with different radii $r$ in Example 1: (a) $r=0.005$ (b) $r=0.003$ (c) $r=0.001$.} 
  \label{fig:disk_r} 
\end{figure}

\begin{figure}
  \centering
  \subfigure[]{\includegraphics[clip, trim = 0 220 0 230, width=.48\textwidth]{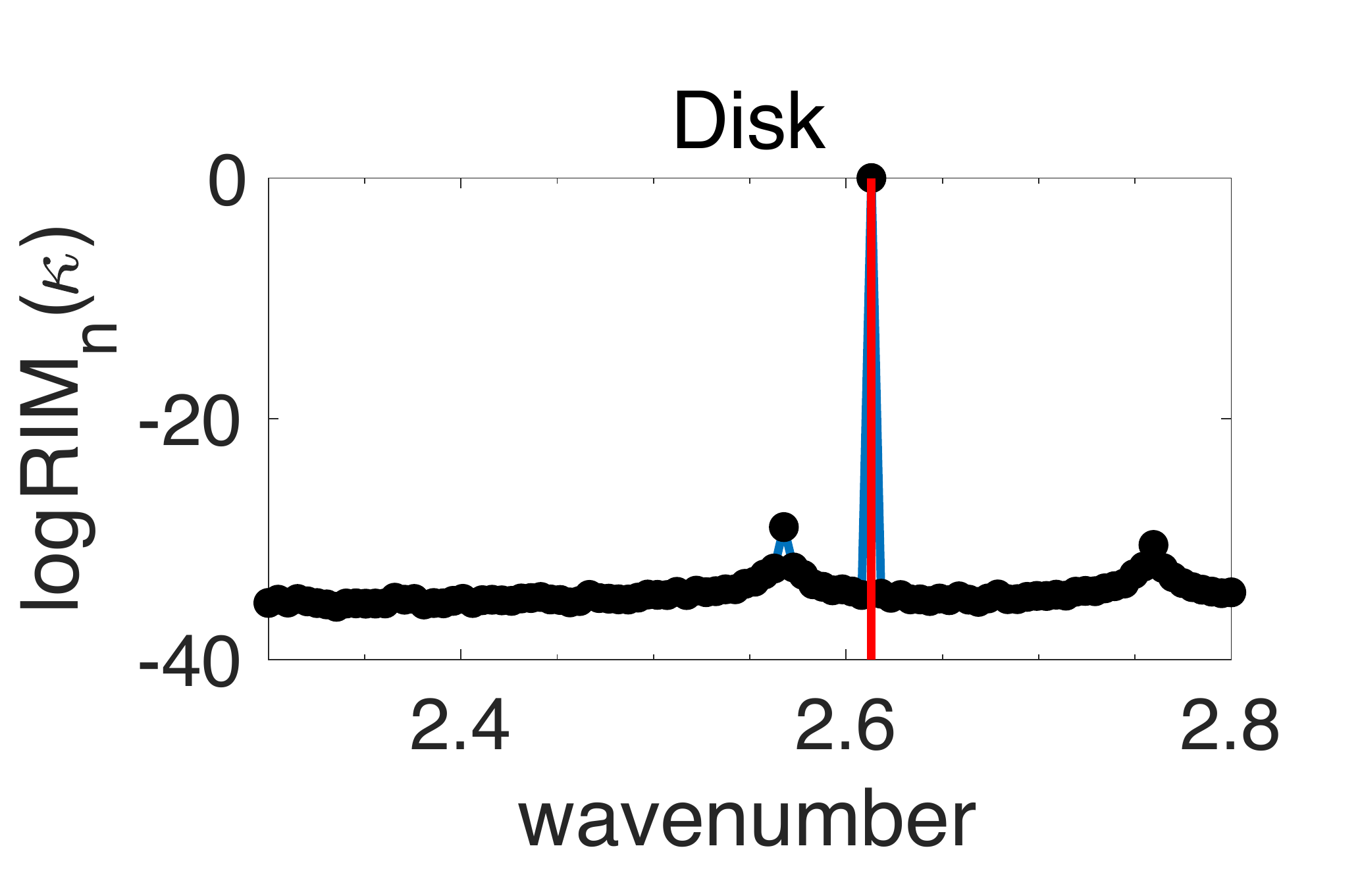}}
  \subfigure[]{\includegraphics[clip, trim = 0 220 0 230, width=.48\textwidth]{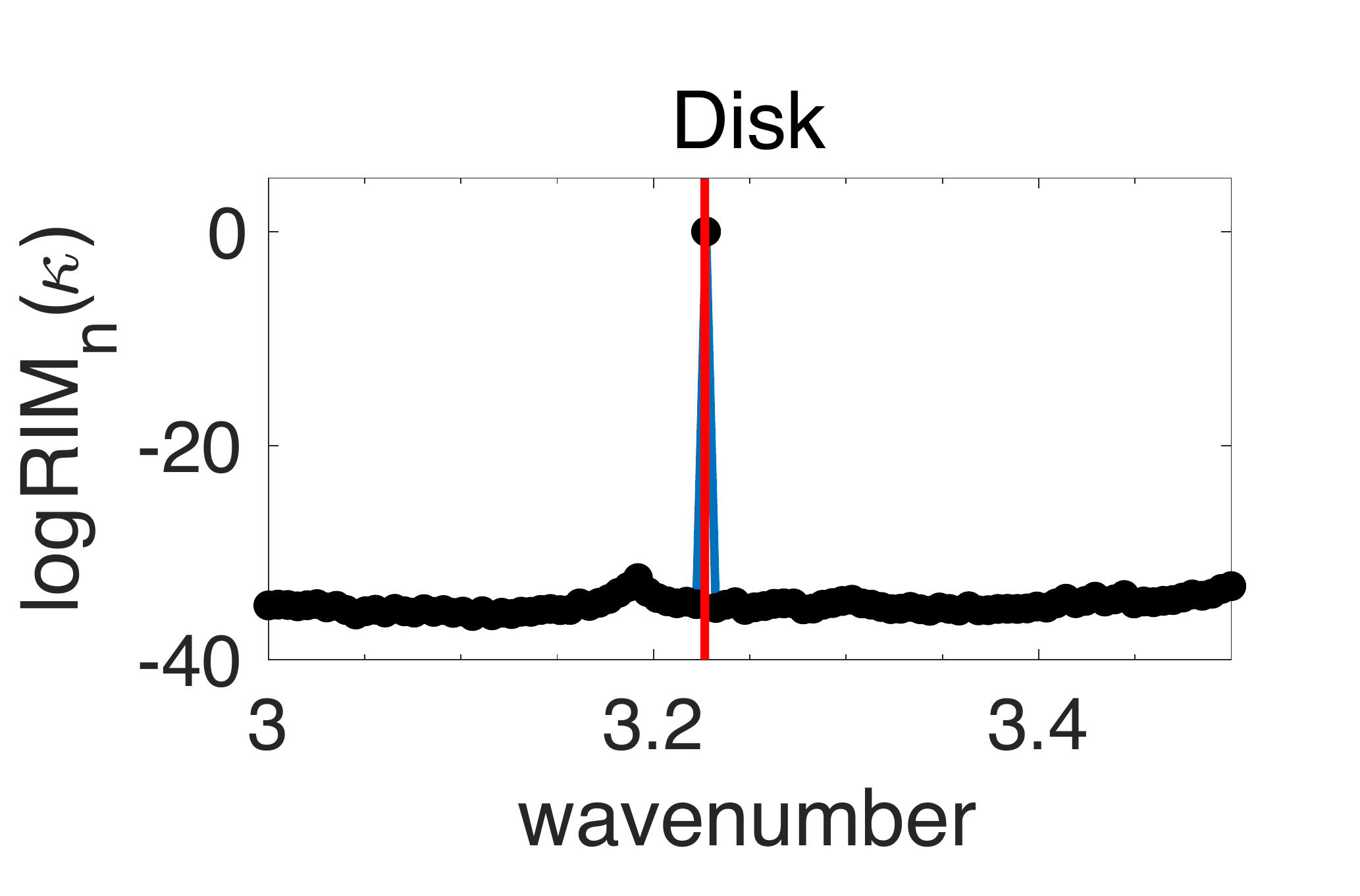}}\\
  \subfigure[]{\includegraphics[clip, trim = 0 220 0 230, width=.48\textwidth]{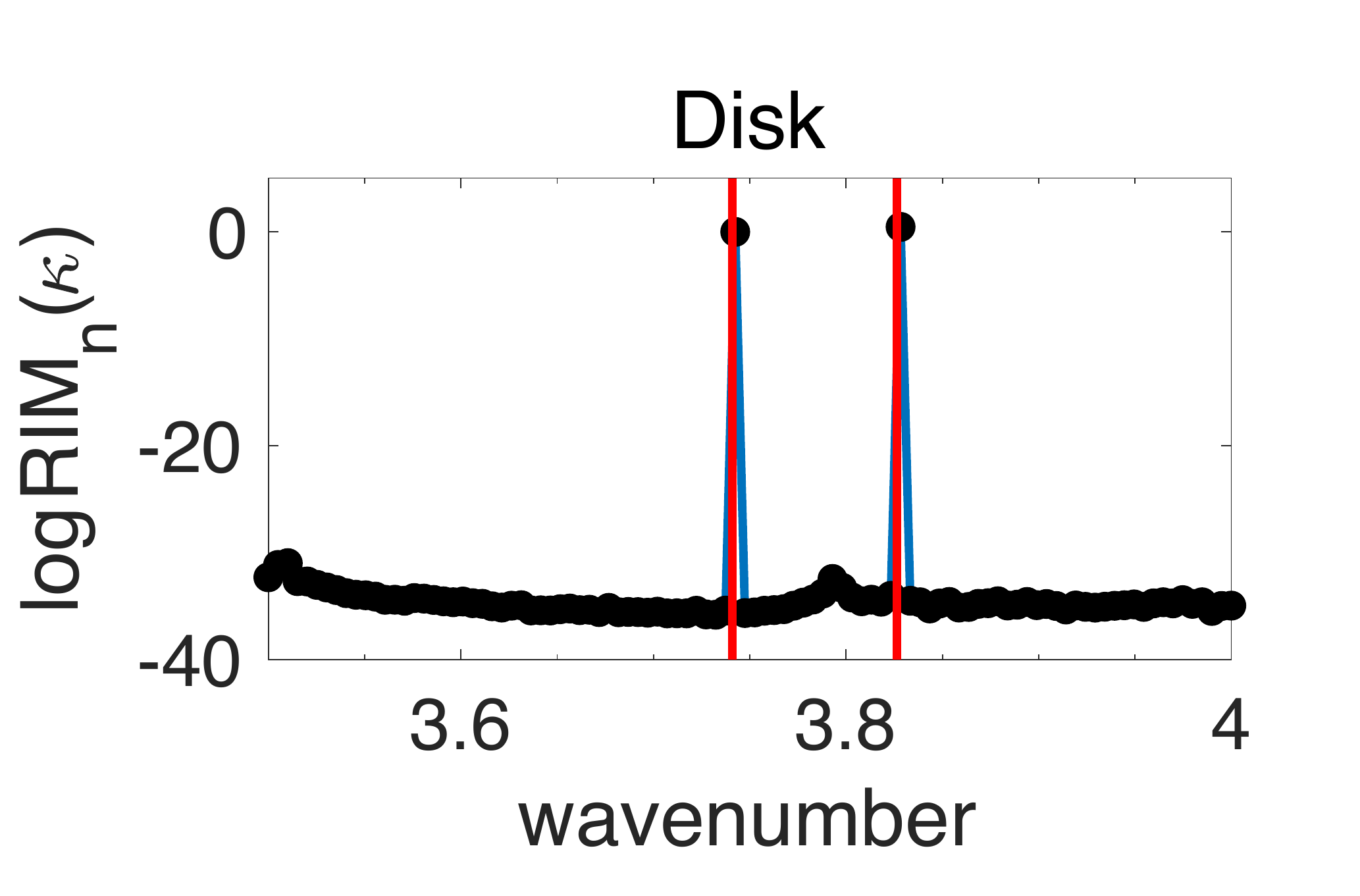}}
  \subfigure[]{\includegraphics[clip, trim = 0 220 0 230, width=.48\textwidth]{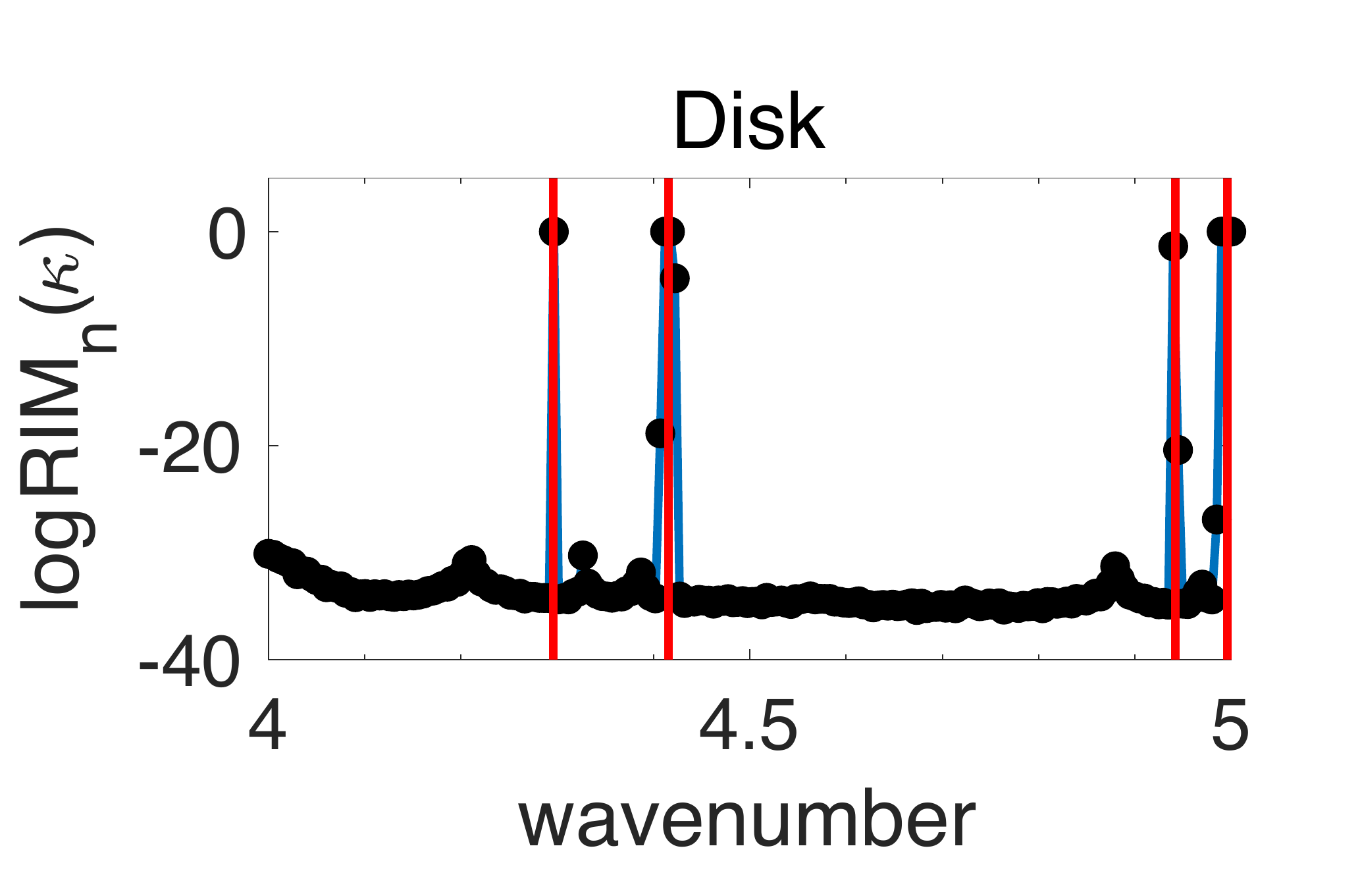}}
  \caption{Plots of $\log{{\rm  RIM}_n(\kappa)}$ in different intervals for Example 1: (a) $[2.3, 2.8]$ (b) $[3, 3.5]$ (c) $[3.5, 4]$ (d) $[4, 5]$.} 
  \label{fig:disk_interval} 
\end{figure}
  
\begin{figure}
  \centering
  \subfigure[]{\includegraphics[clip, trim = 0 200 0 200, width=.48\textwidth]{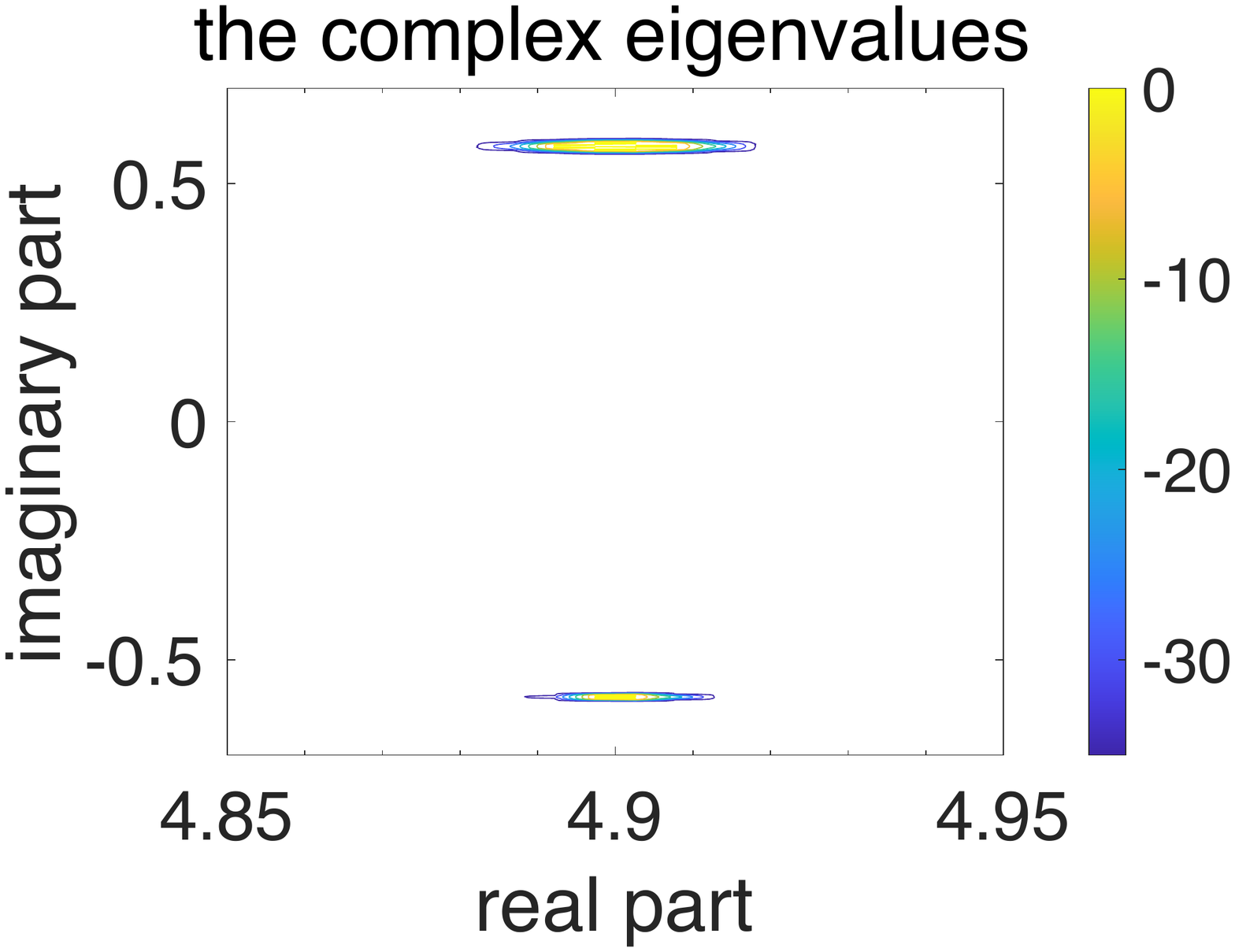}}
  \subfigure[]{\includegraphics[clip, trim = 0 200 0 200, width=.48\textwidth]{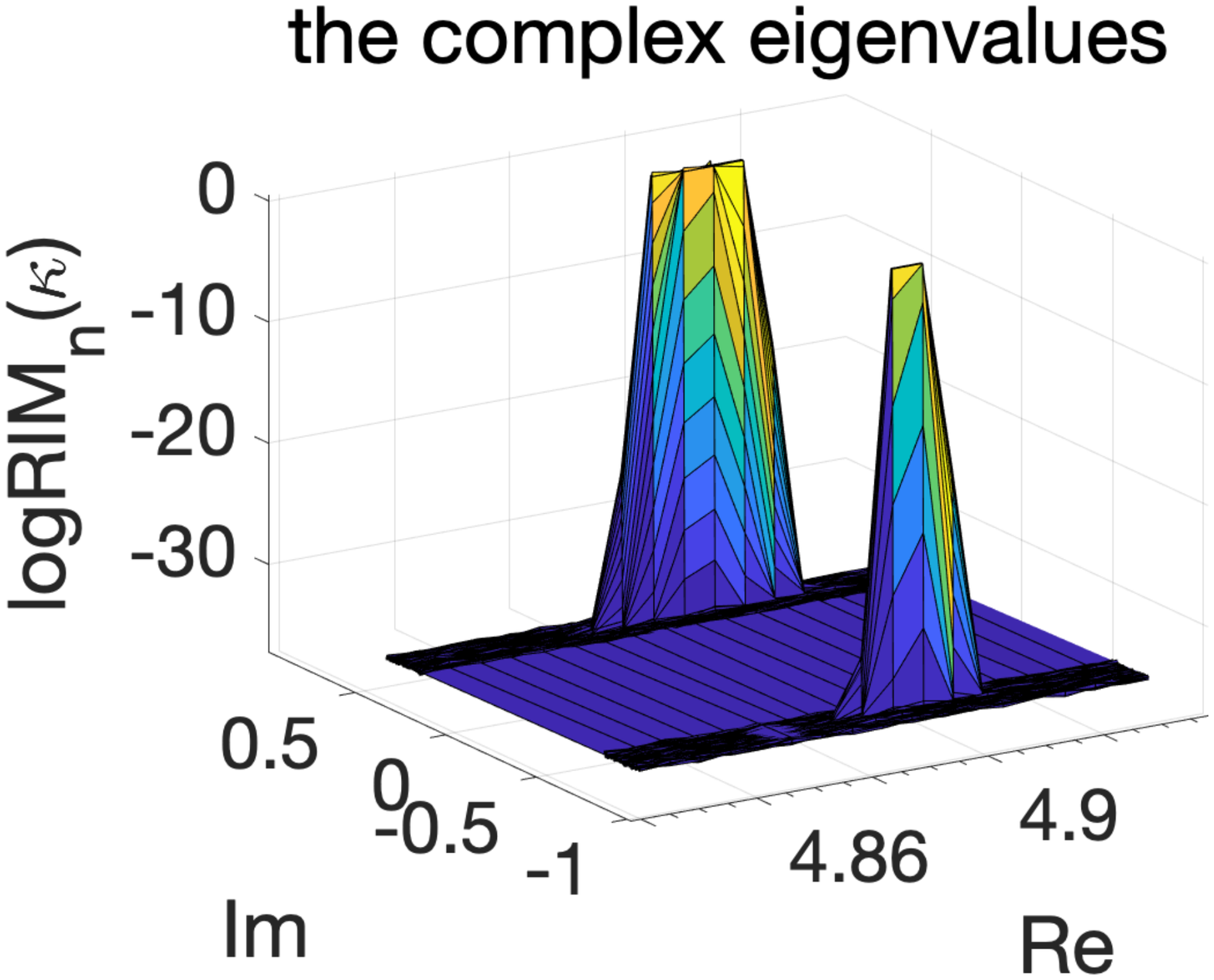}}
 \caption{Plots of $\log{{\rm  RIM}_n(\kappa)}$ for Example 1 in the complex plane: (a) contour plot (b) surface plot.} 
  \label{fig:disk_complex} 
\end{figure}

\begin{example} 
We consider in this example a peanut-shaped domain enclosed by the equation
$$
\sqrt{0.25+\cos^2 t}(\cos t, \sin t), \quad t\in[0, 2\pi].
$$
We choose the interval to be $[1.3, 1.6]$ and  $[1.65, 2]$ with dividing it into $200$ subintervals. The results are shown in Figure \ref{fig:peanut}. We compare the results with the eigenvalues computed by finite element methods from \cite{Tiexiang2015}. Each location of these eigenvalues is marked by a red line. We conclude that the algorithm proposed in this paper is effective.
\end{example}

\begin{figure}
  \centering
  \subfigure[]{\includegraphics[clip, trim = 0 200 0 200, width=.48\textwidth]{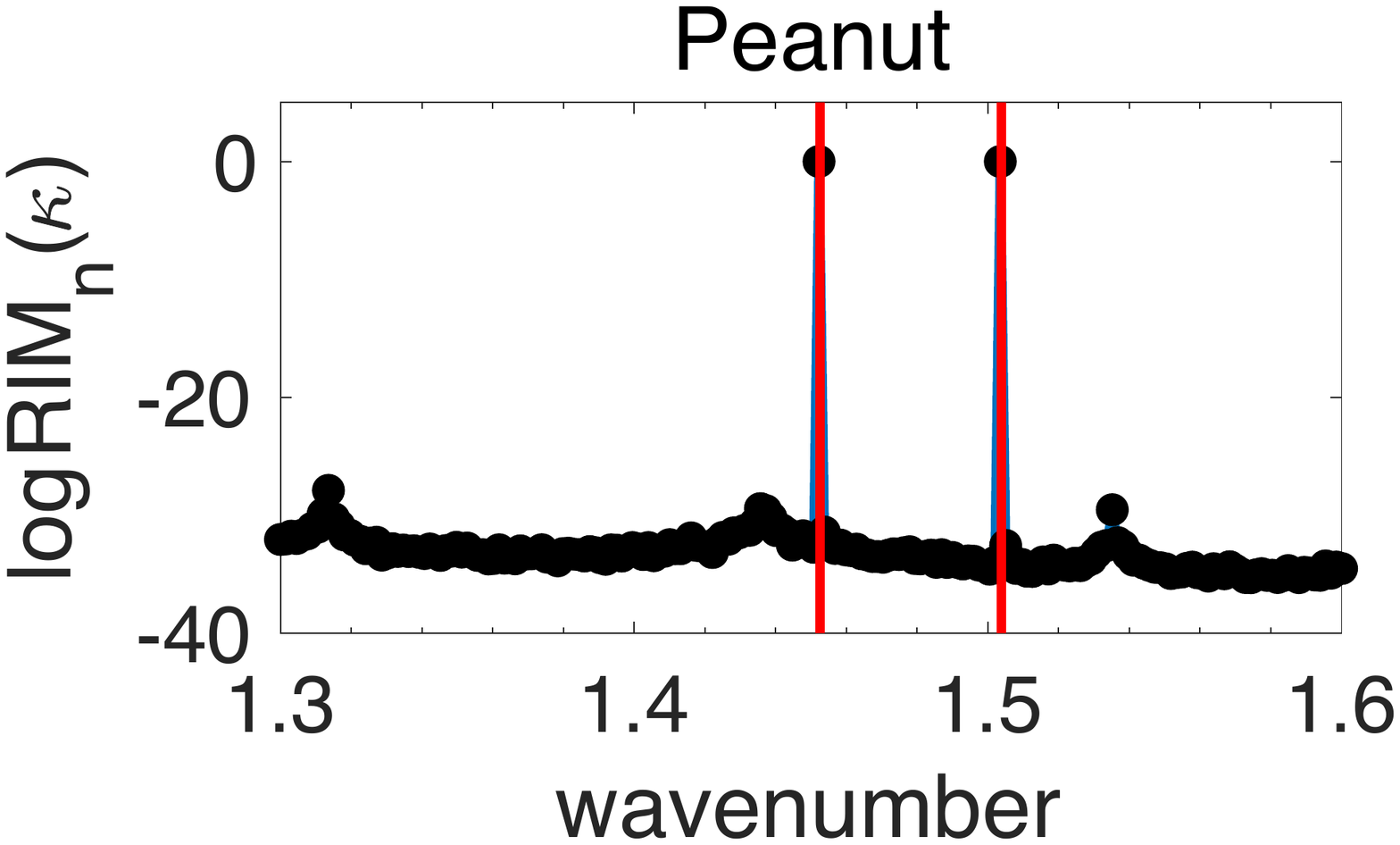}}
  \subfigure[]{\includegraphics[clip, trim = 0 200 0 200, width=.48\textwidth]{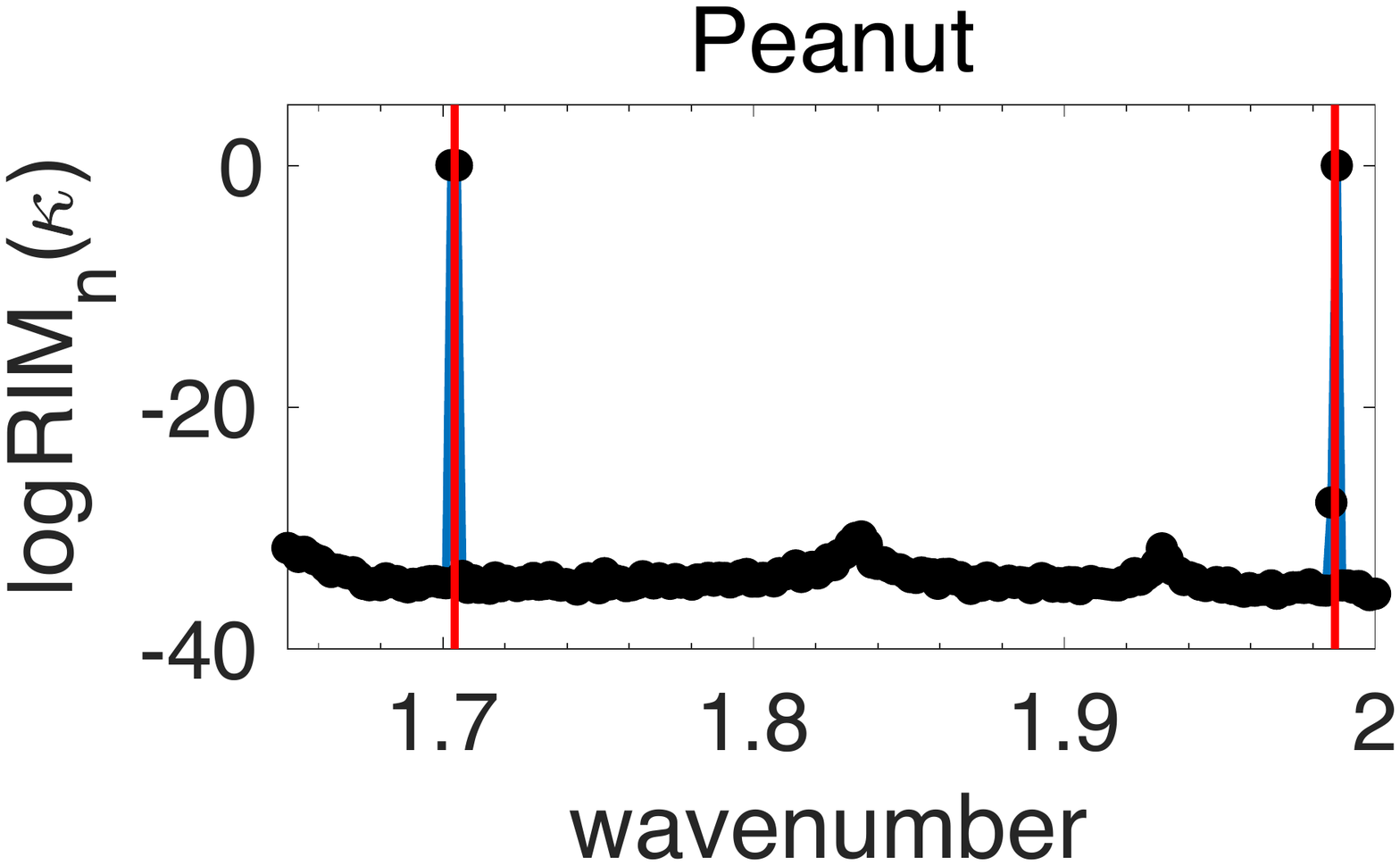}}
  \caption{Plots of $\log{{\rm  RIM}_n(\kappa)}$ for Example 2 in different intervals: (a)  $[1.3, 1.6]$ (b) $[1.65,2]$.} 
  \label{fig:peanut} 
\end{figure}

In the following four examples, we test the method with regularization for some domains with corners, 
where the domains are plotted in Figure \ref{fig:shapes_with_corner}. We note that  the location of
the eigenvalues computed by finite element methods from \cite{Tiexiang2015} is marked by a red line
in Example 3 and 4.

\begin{figure}
  \centering
  \subfigure[]{\includegraphics[clip, trim = 0 200 0 200, width=.23\textwidth]{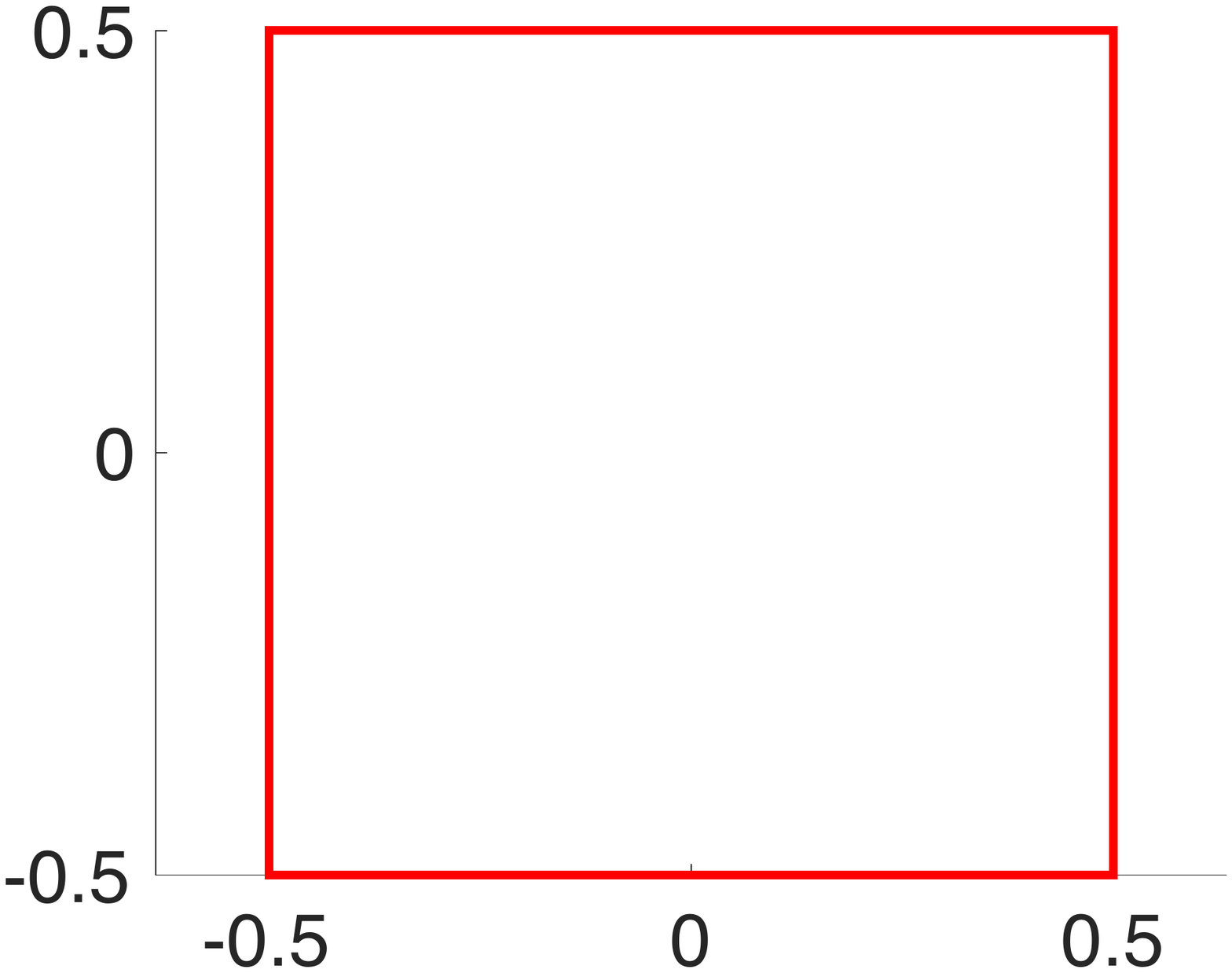}}
  \subfigure[]{\includegraphics[clip, trim = 0 200 0 200, width=.23\textwidth]{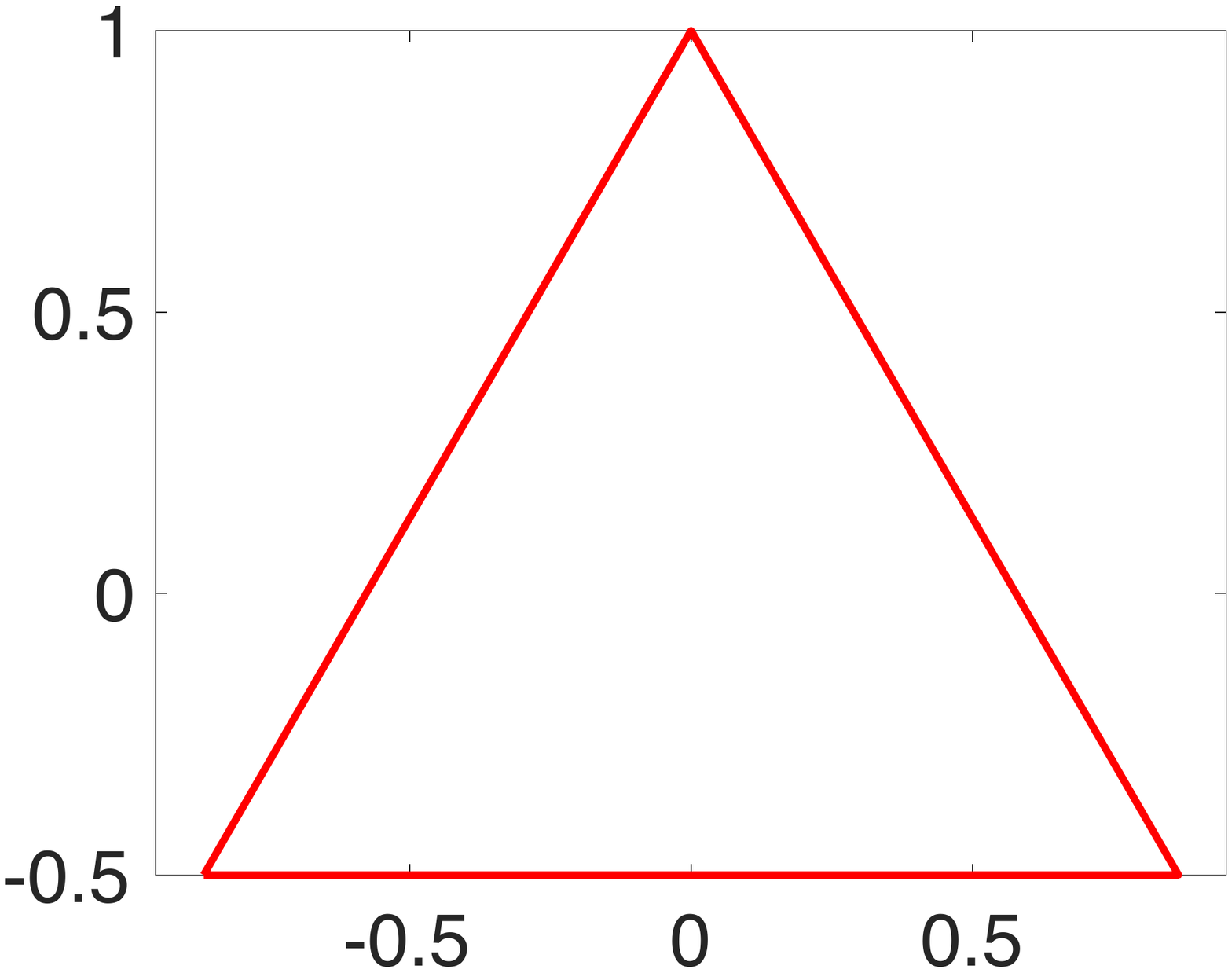}}
  \subfigure[]{\includegraphics[clip, trim = 0 200 0 200, width=.23\textwidth]{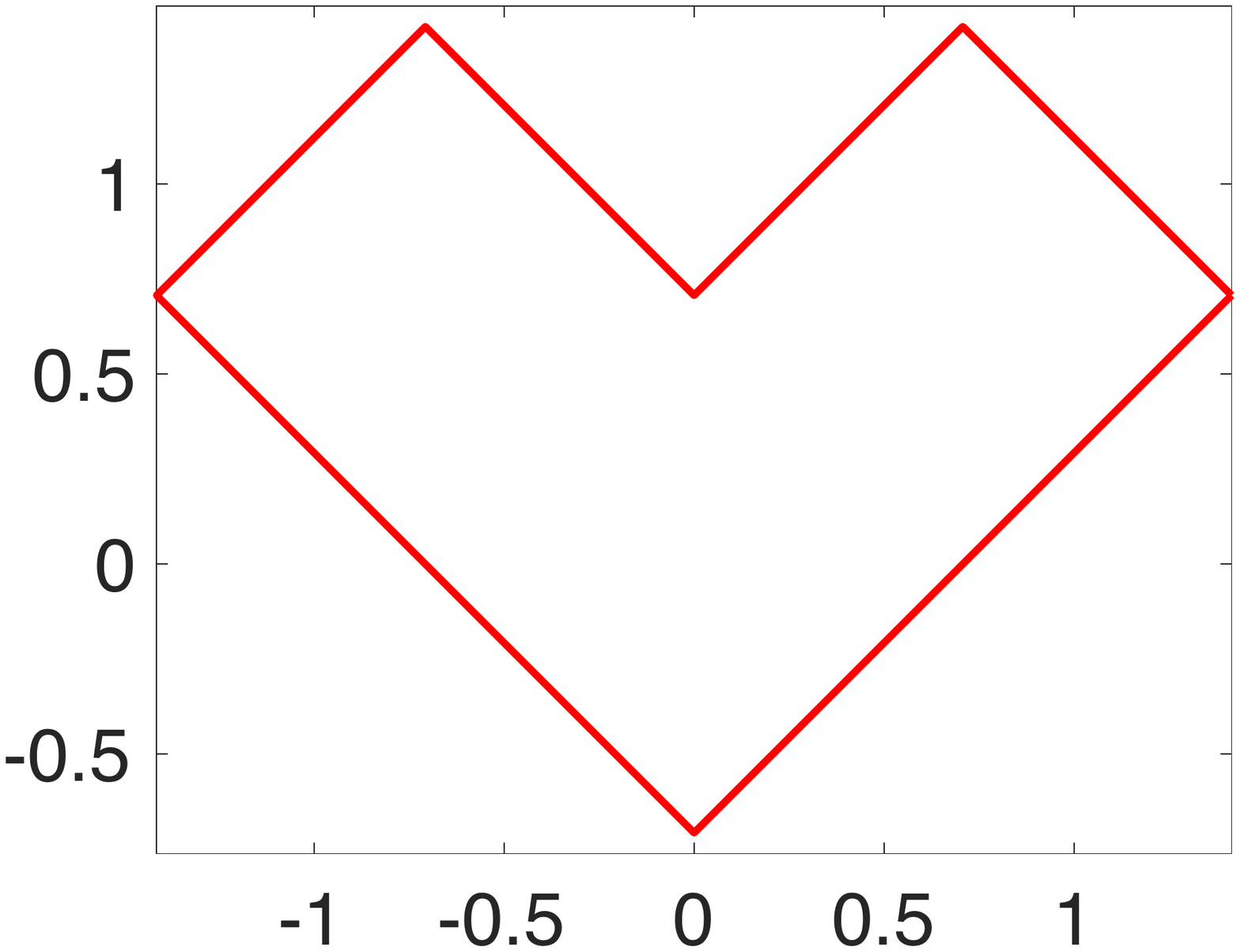}}
  \subfigure[]{\includegraphics[clip, trim = 0 200 0 200, width=.23\textwidth]{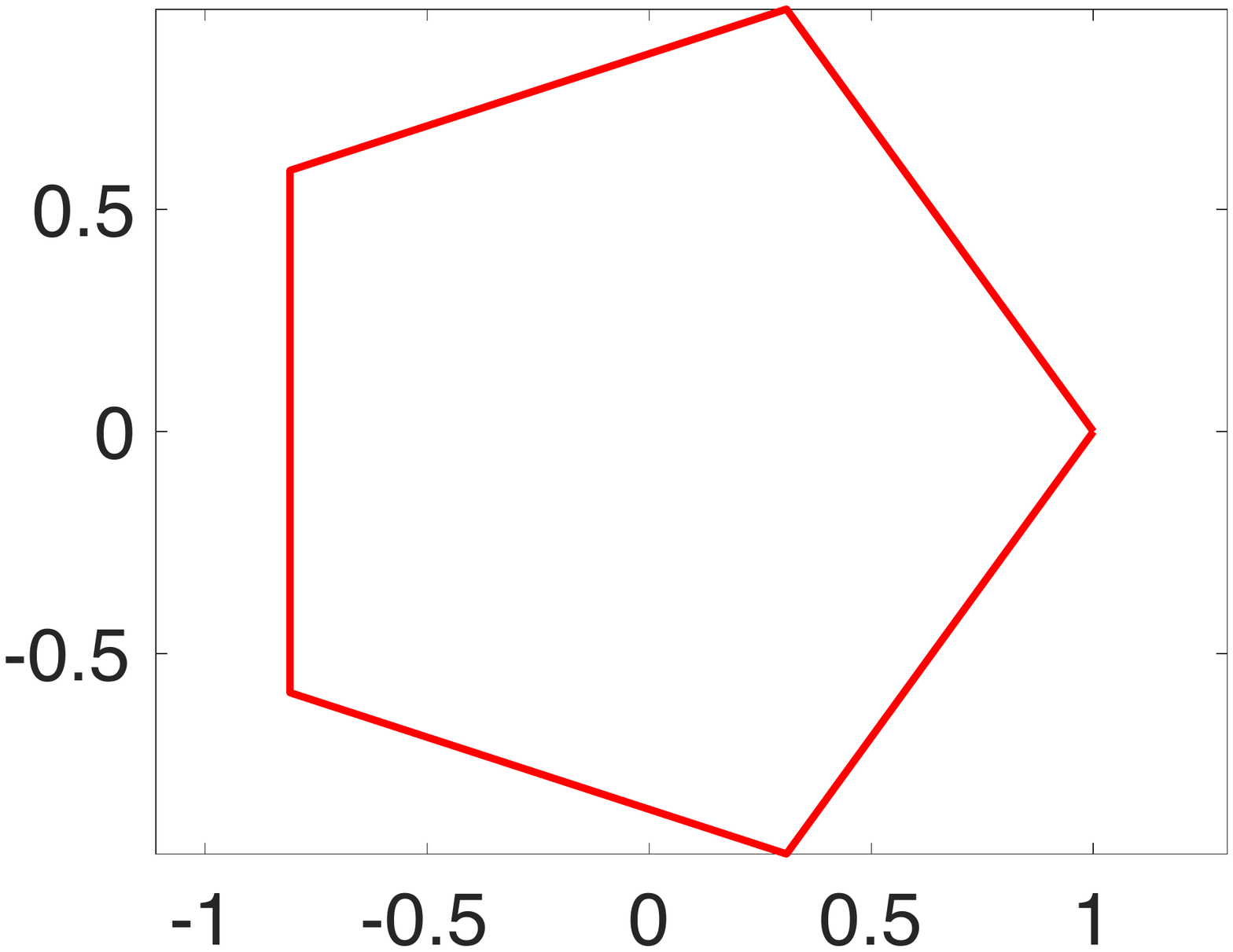}}
  \caption{Several domains with corners: (a) square (b) triangle (c) L-shape (d) pentagon.} 
  \label{fig:shapes_with_corner} 
\end{figure}

\begin{example} 
We consider in this example a unit square centered at the origin. We choose the intervals to be $[1.65, 1.95]$, $[2.25,2.55]$ and $[2.7, 3]$. Each interval is divided  into $100$ subintervals. The results are shown in Figures \ref{fig:square}-\ref{fig:square_interval}.  We plot the results in Figure \ref{fig:square} for different regularization parameters  with $\eta=10^{-m}, m=2,3,\dots,7$. We see that the result for $\eta=10^{-2}$ and $\eta=10^{-7}$ is worst. This implies that the regularization is indeed necessary and effective. According to Figure \ref{fig:square}, we shall use $\eta=10^{-5}$ as the regularization parameter in the following examples. The results for the intervals  $[2.25, 2.55]$ and $[2.7, 3]$ are shown in Figure \ref{fig:square_interval}.
\end{example}

\begin{figure}
	\centering
	\subfigure[]{\includegraphics[clip, trim = 0 220 0 230, width=.48\textwidth]{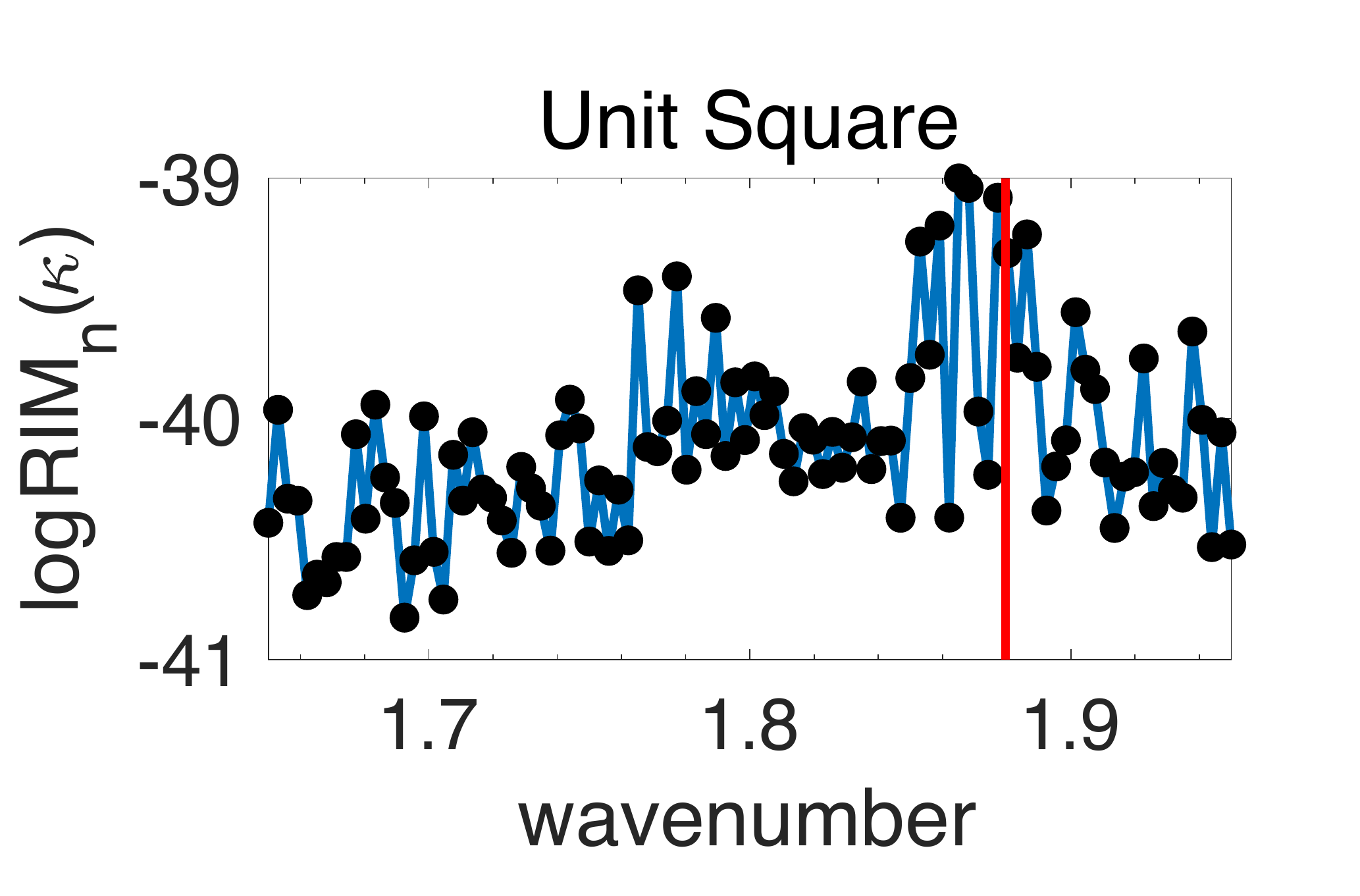}}
	\subfigure[]{\includegraphics[clip, trim = 0 220 0 230, width=.48\textwidth]{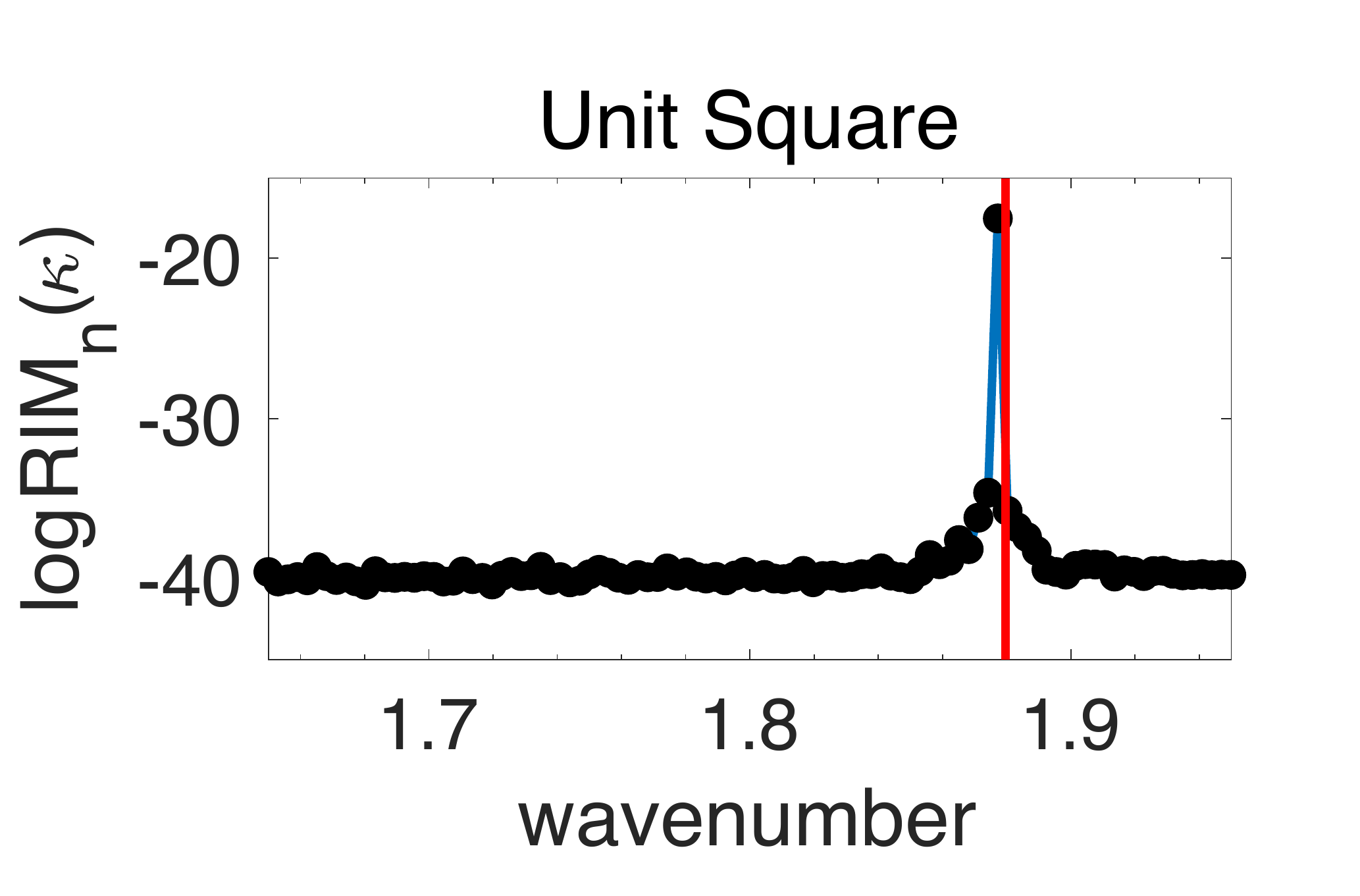}}\\
	\subfigure[]{\includegraphics[clip, trim = 0 220 0 230, width=.48\textwidth]{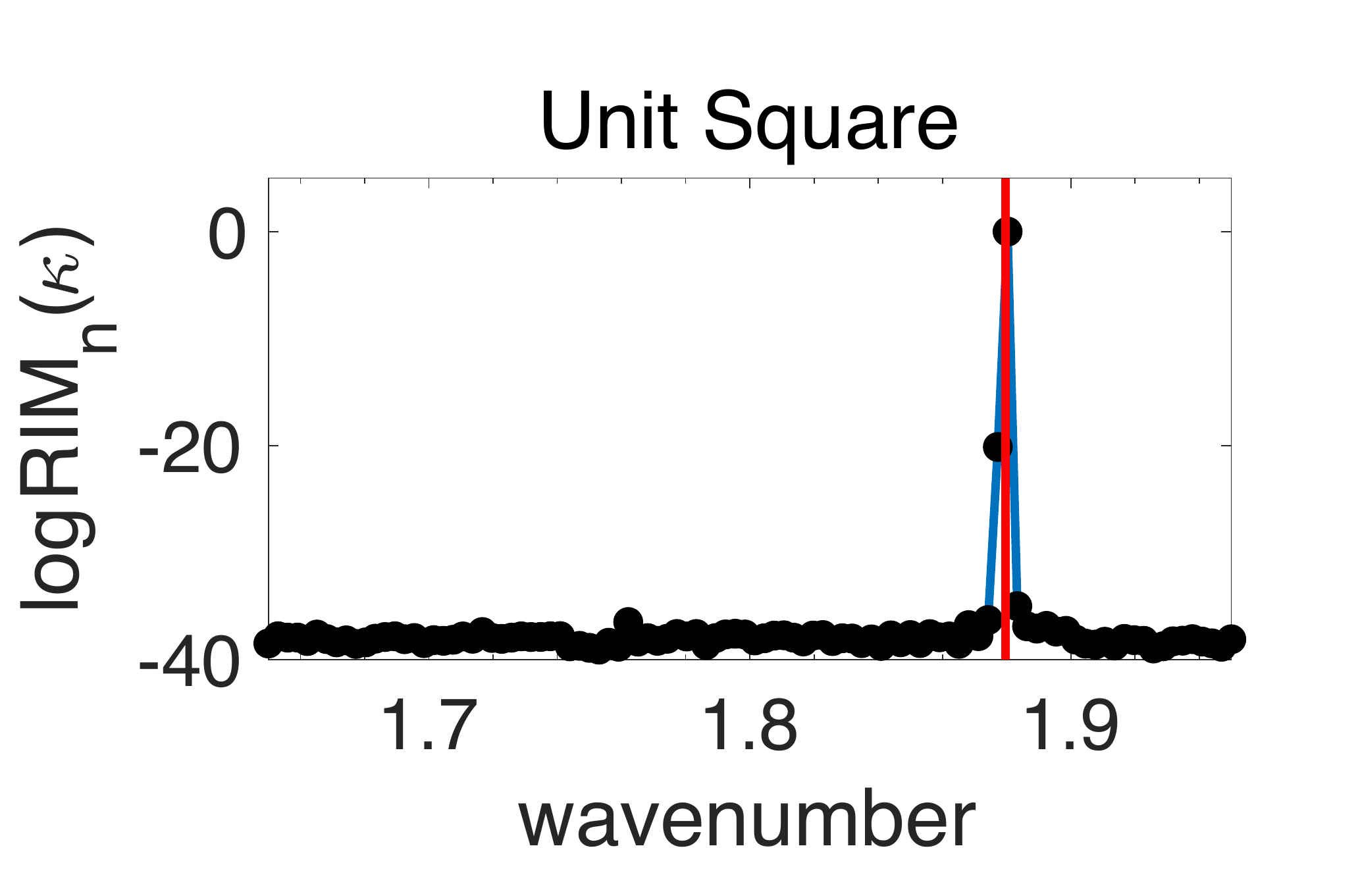}}
	\subfigure[]{\includegraphics[clip, trim = 0 220 0 230, width=.48\textwidth]{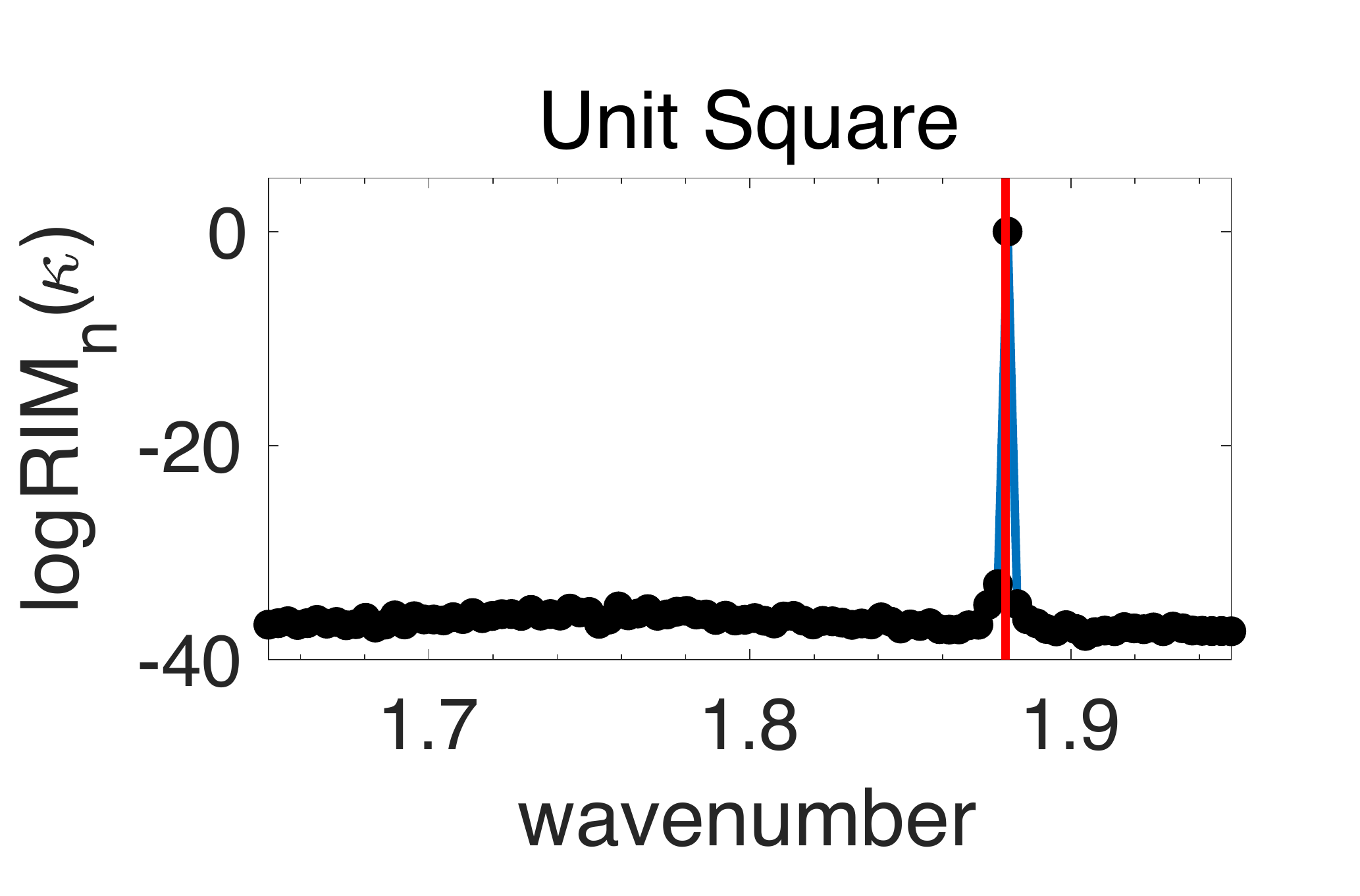}} \\
	\subfigure[]{\includegraphics[clip, trim = 0 220 0 230, width=.48\textwidth]{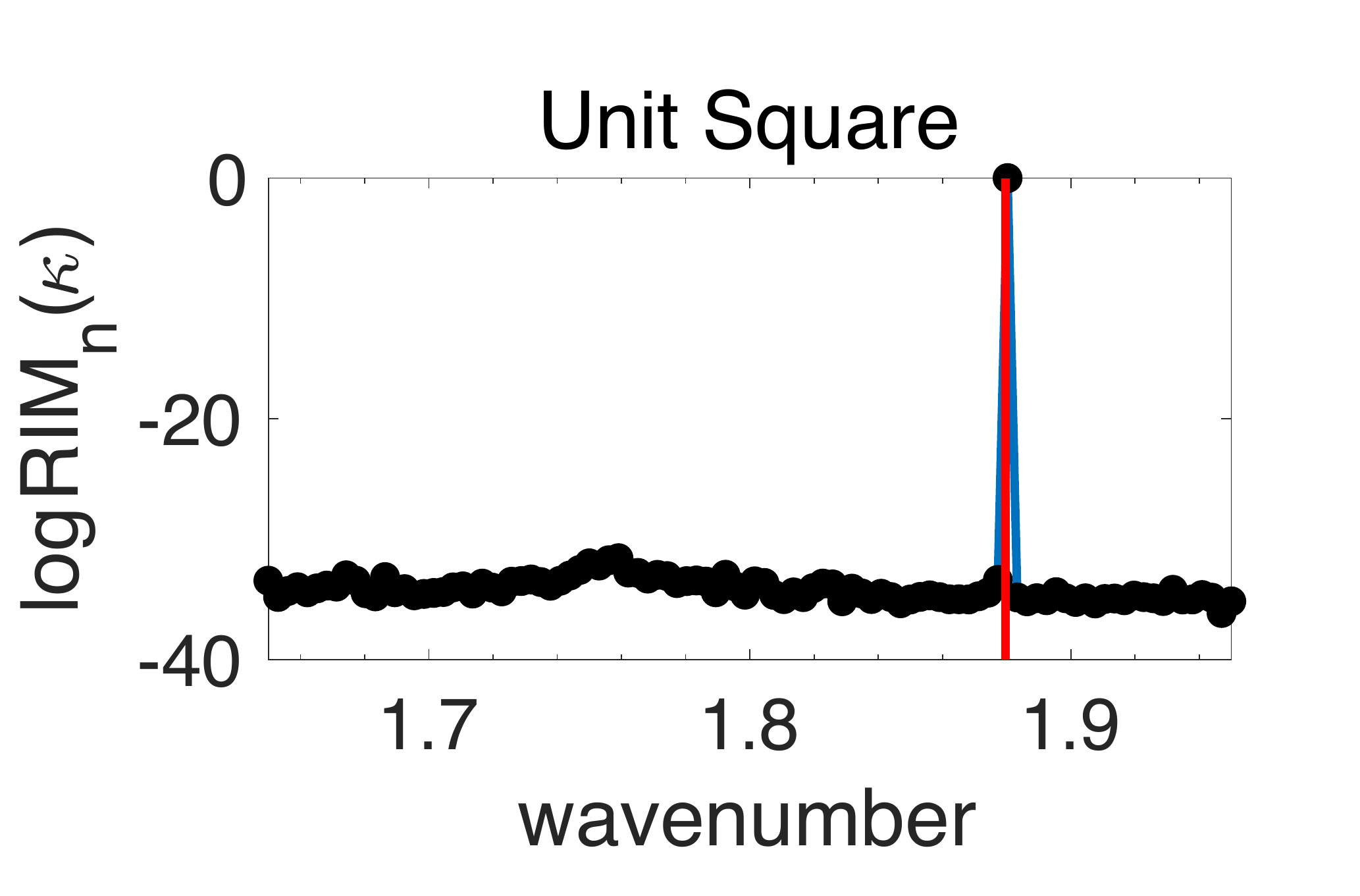}} 
	\subfigure[]{\includegraphics[clip, trim = 0 220 0 230, width=.48\textwidth]{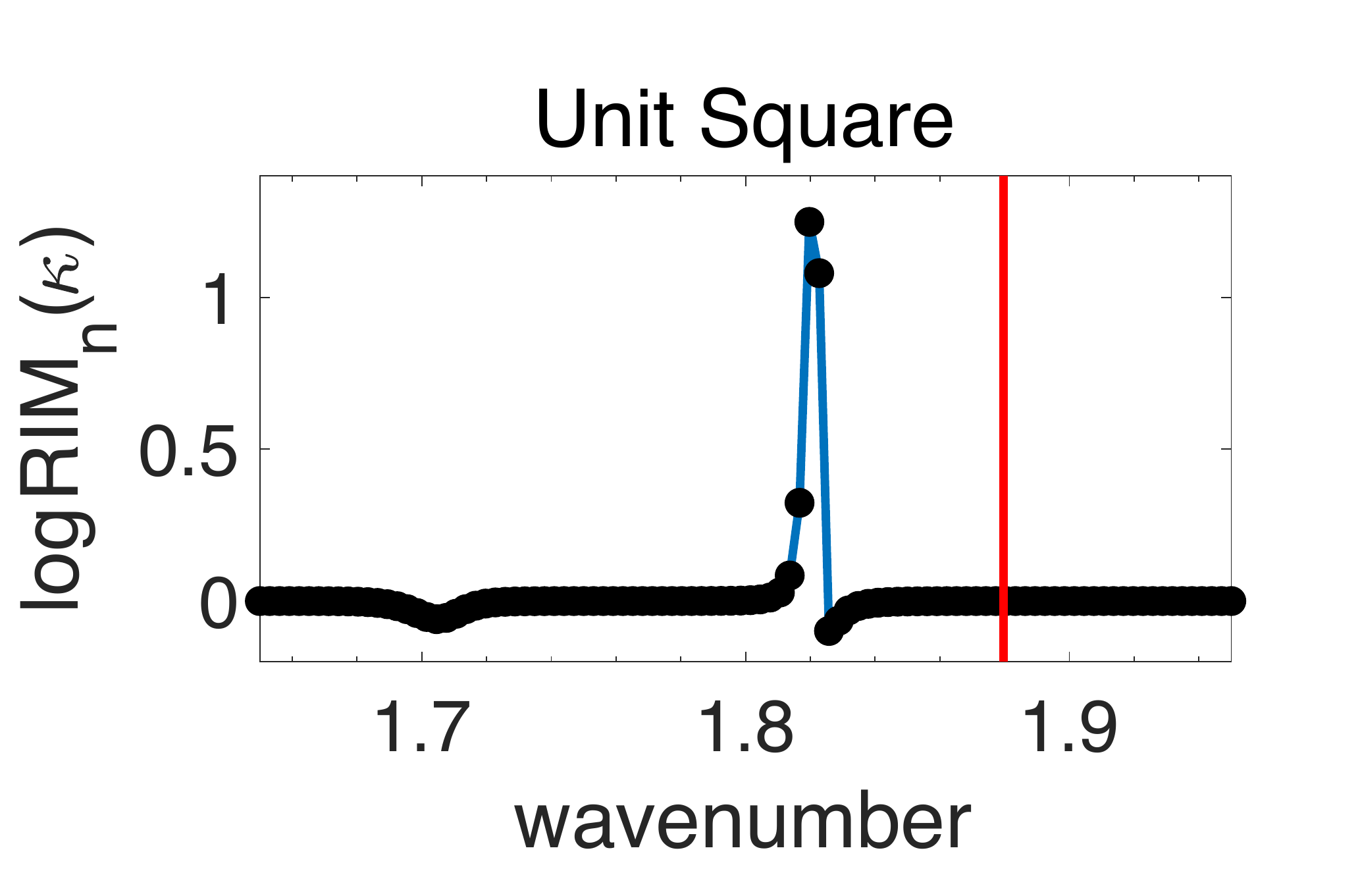}} 
	\caption{Plots of $\log{{\rm  RIM}_n(\kappa)}$ for different $\eta$ in Example 3: (a)  $\eta=10^{-2}$ (b) $\eta=10^{-3}$ (c) $\eta=10^{-4}$ (d) $\eta=10^{-5}$ (e) $\eta=10^{-6}$ (g) $\eta=10^{-7}$.} 
	\label{fig:square} 
\end{figure}

\begin{figure}
  \centering
  \subfigure[]{\includegraphics[clip, trim = 0 220 0 230, width=.48\textwidth]{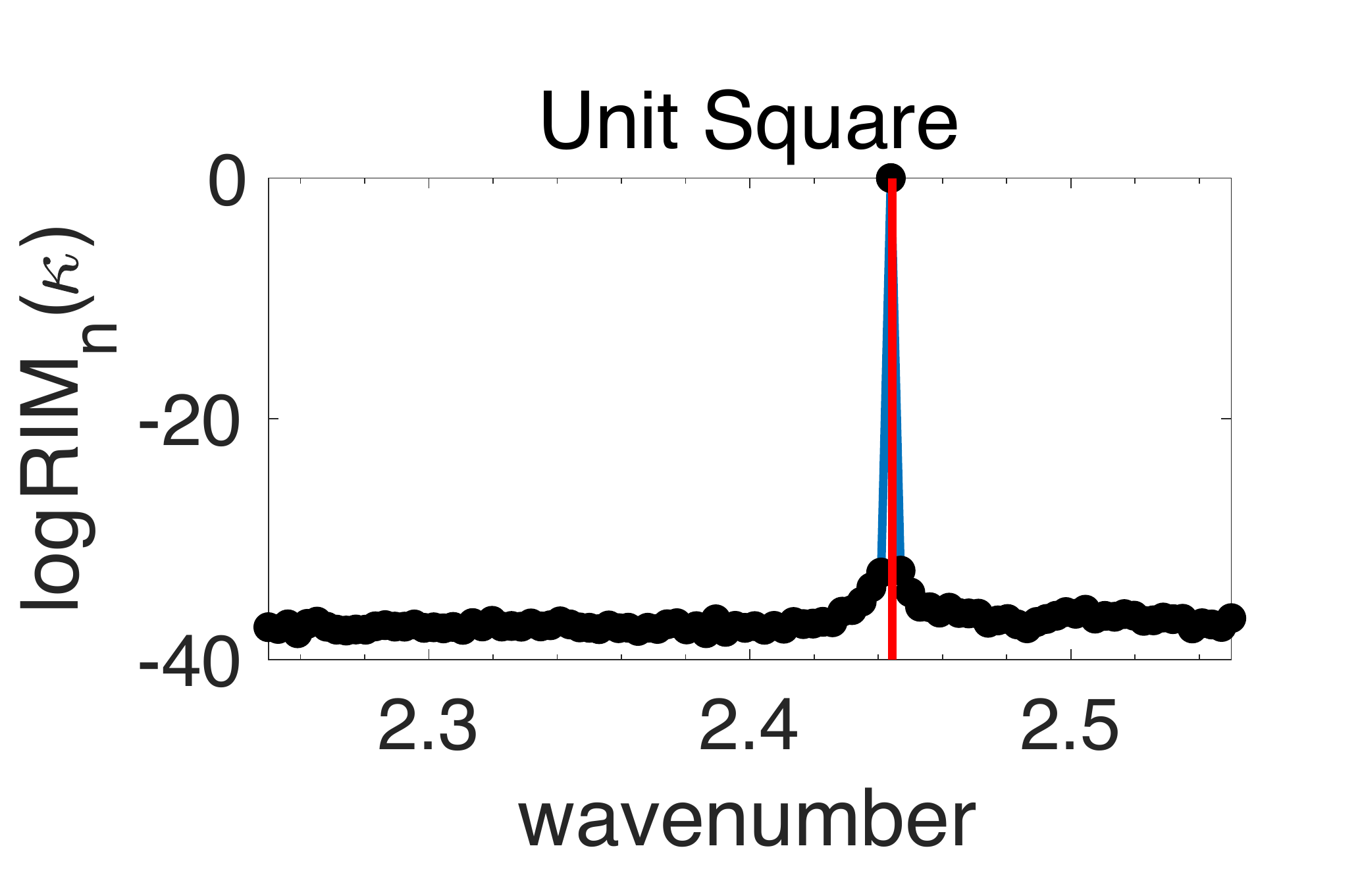}}
  \subfigure[]{\includegraphics[clip, trim = 0 220 0 230, width=.48\textwidth]{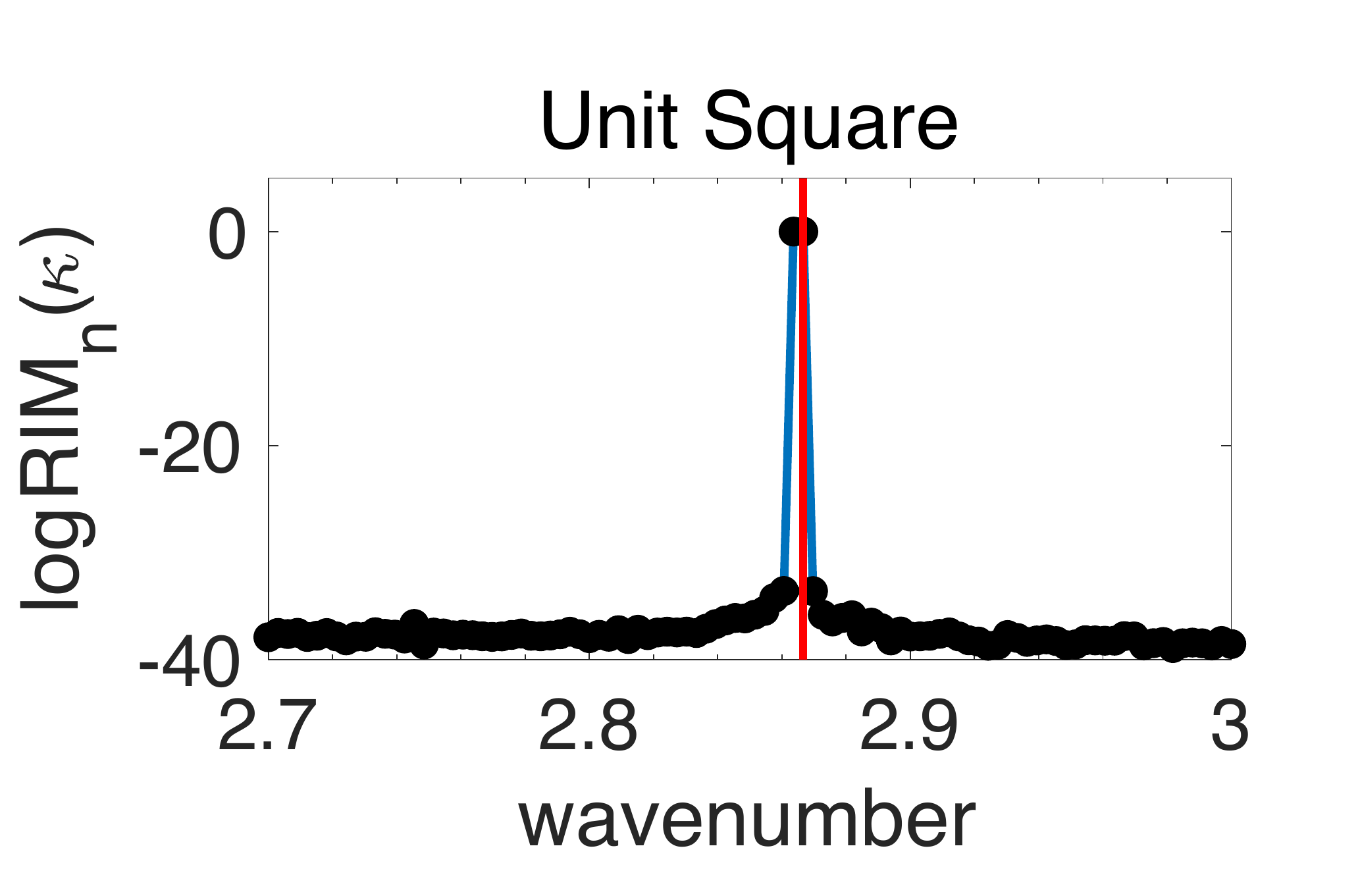}}
  \caption{Plots of $\log{{\rm  RIM}_n(\kappa)}$ for Example 3 in different intervals: (a) $[2.25, 2.55]$ (b) $[2.7, 3]$.} 
  \label{fig:square_interval} 
\end{figure}

\begin{example} 
We consider in this example a triangle with vertexes $(-\sqrt{3}/2,-1/2)$, $(\sqrt{3}/2,-1/2)$ and $(0,1)$.
We choose the intervals to be $[1.65, 1.95]$, $[2.10,2.40]$ and $[2.70,3.00]$.
Each interval is divided  into $100$ subintervals. The results are shown in Figure \ref{fig:triangle}.  
\end{example}

\begin{figure}
  \centering
  \subfigure[]{\includegraphics[clip, trim = 0 220 0 230, width=.48\textwidth]{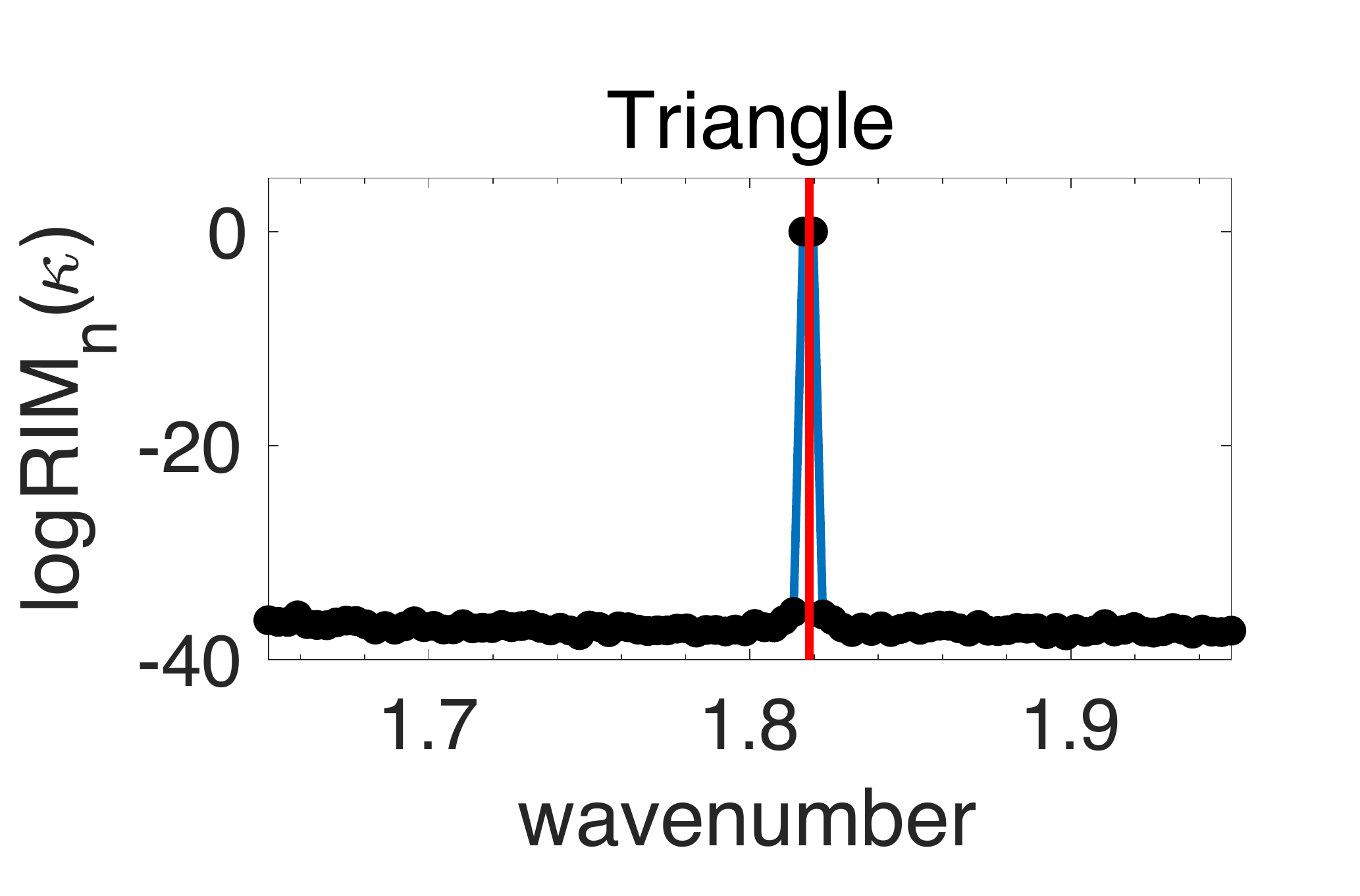}}
  \subfigure[]{\includegraphics[clip, trim = 0 220 0 230, width=.48\textwidth]{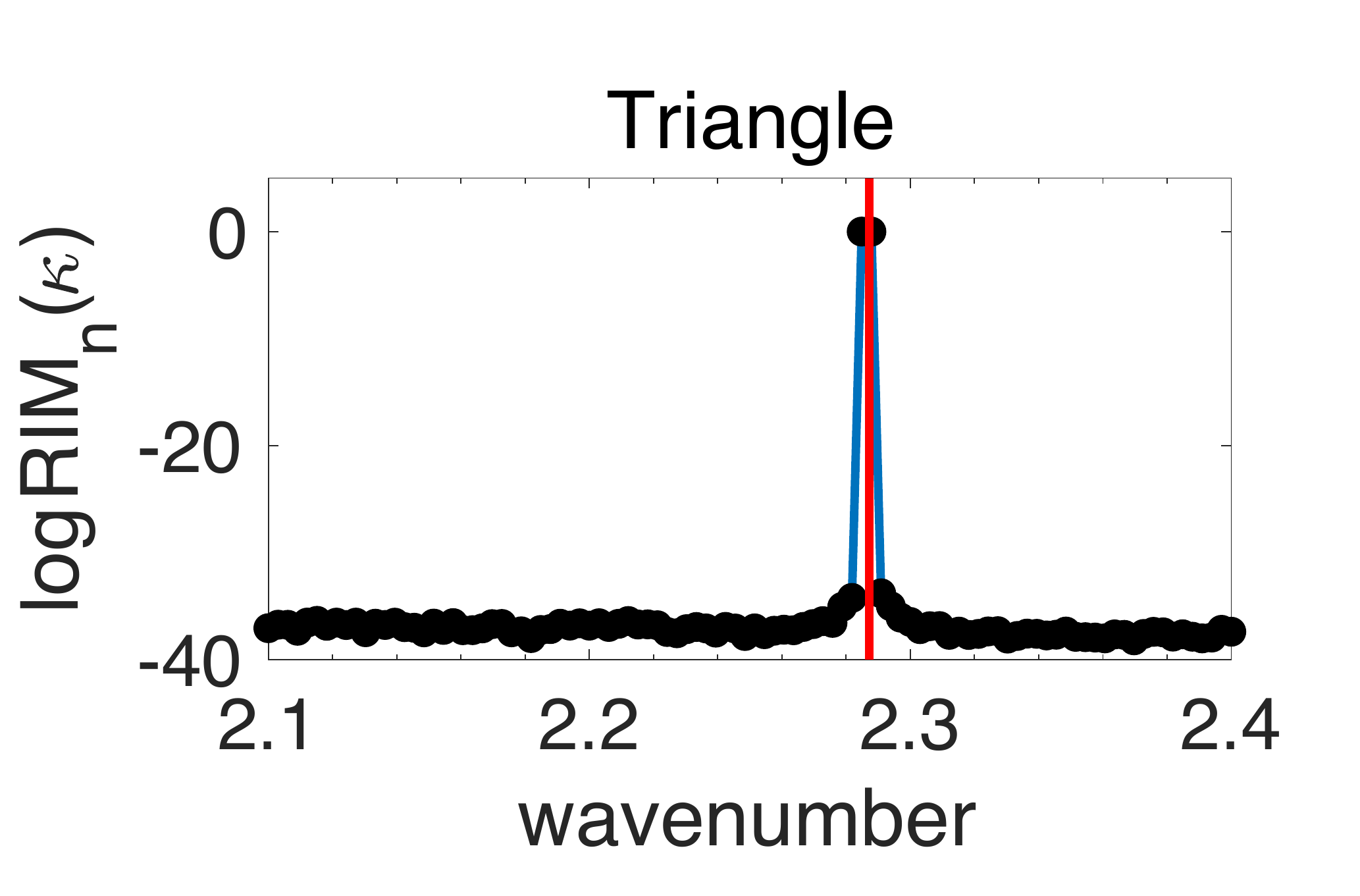}}\\
  \subfigure[]{\includegraphics[clip, trim = 0 220 0 230, width=.48\textwidth]{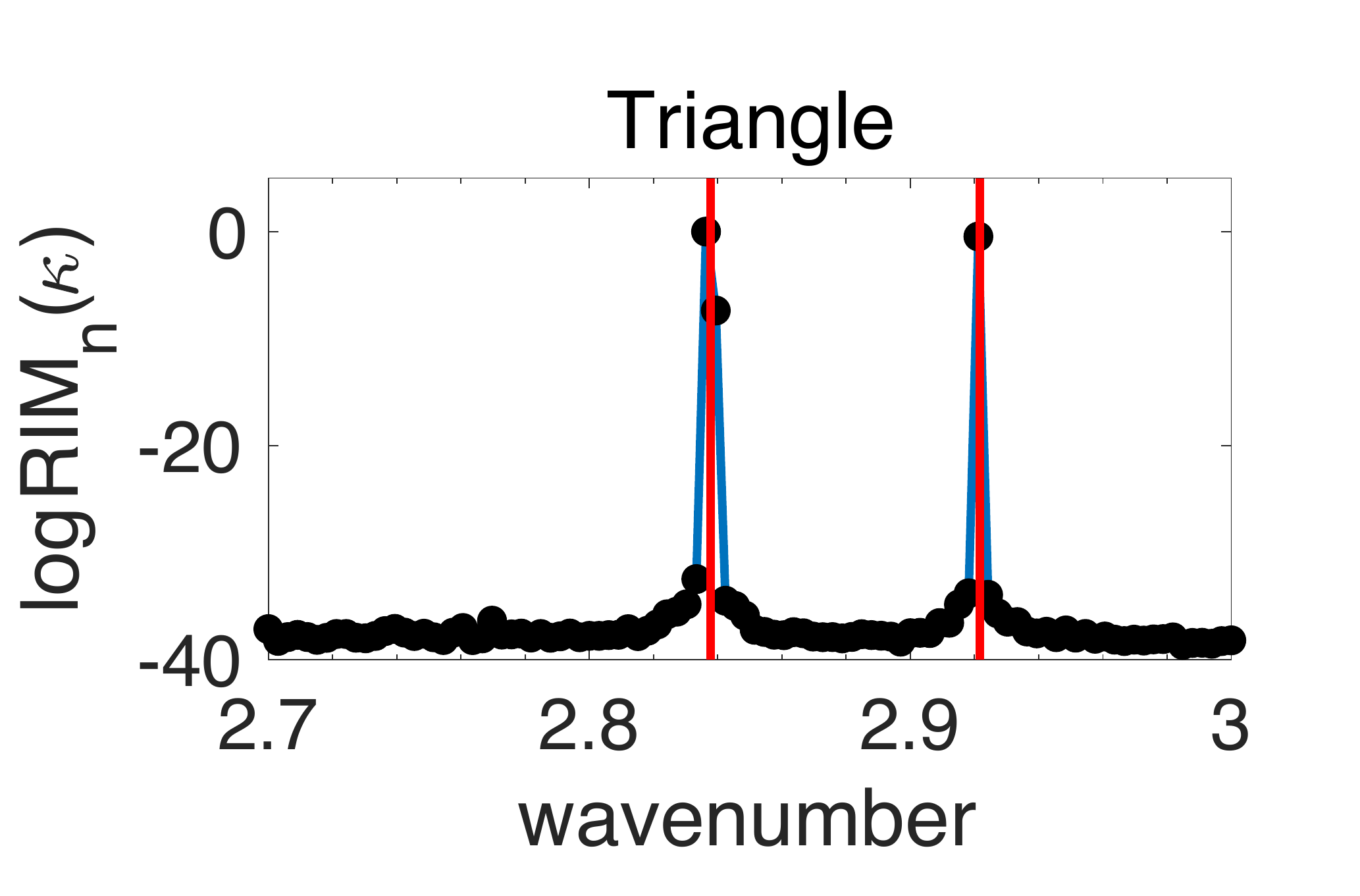}}
  \caption{Plots of $\log{{\rm  RIM}_n(\kappa)}$ for Example 4 in different intervals: (a) $[1.65, 1.95]$ (b) $[2.1, 2.4]$ (c)  $[2.7, 3]$.} 
  \label{fig:triangle} 
\end{figure}

\begin{example} 
We consider in this example the L-shape domain with vertexes $(\sqrt{2},\sqrt{2}/{2})$, $(\sqrt{2}/2,\sqrt{2})$, $(0,\sqrt{2}/2)$, $(-\sqrt{2}/2,\sqrt{2})$, $(-\sqrt{2},\sqrt{2}/{2})$ and  $(0,-\sqrt{2}/2)$.
We choose the intervals to be $[1.5, 2]$, $[2, 2.5]$ and $[2.5, 3]$.
Each interval is divided  into $200$ subintervals. The results are shown in Figure \ref{fig:L-shape}. 
The eigenvalues computed in this domains are $1.5541$, $1.8920$, $2.1131$, $2.4937$, $2.5226$, $2.7349$.
\end{example}
 
\begin{figure}
  \centering
  \subfigure[]{\includegraphics[clip, trim = 0 220 0 230, width=.48\textwidth]{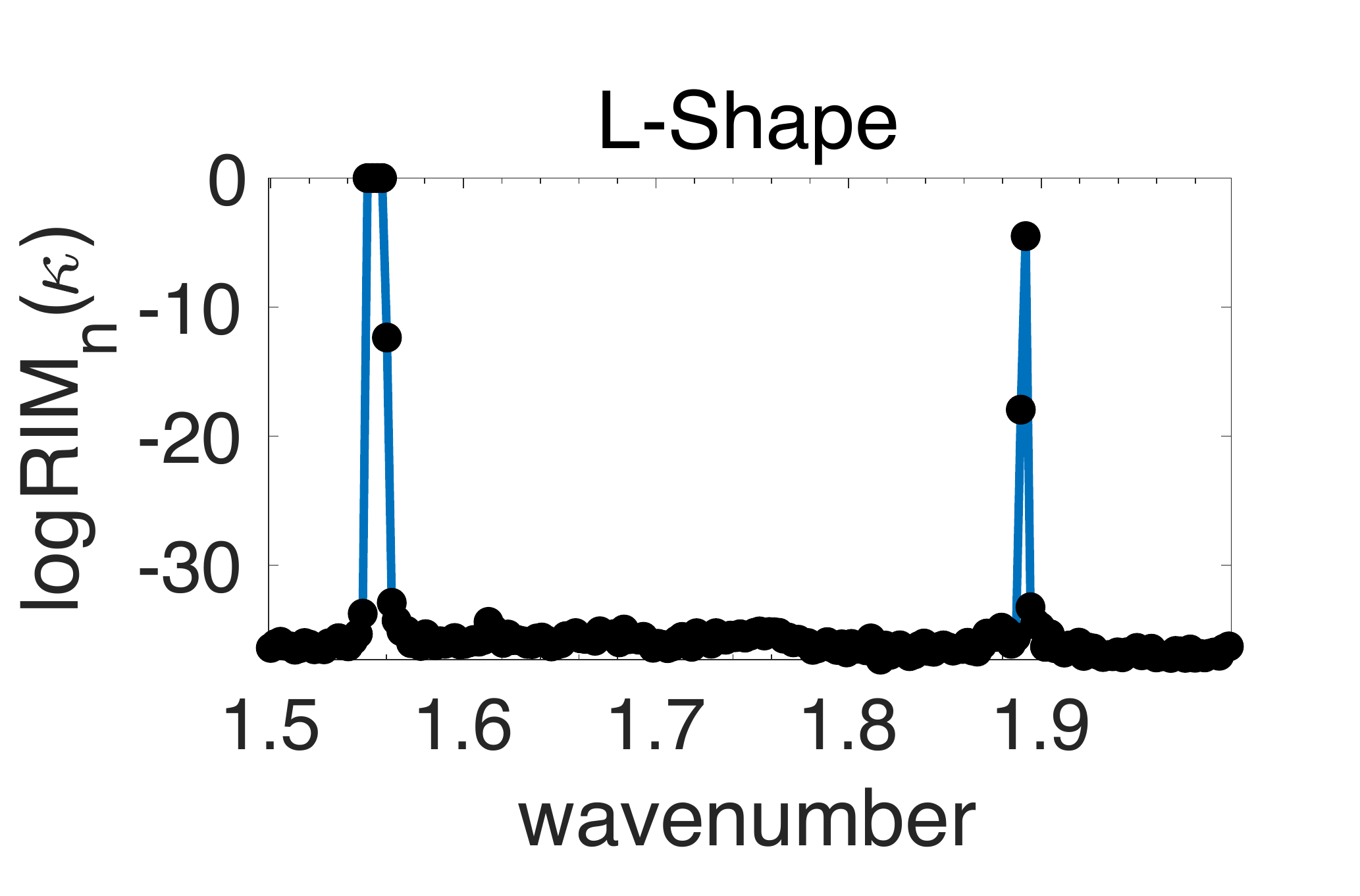}}
  \subfigure[]{\includegraphics[clip, trim = 0 220 0 230, width=.48\textwidth]{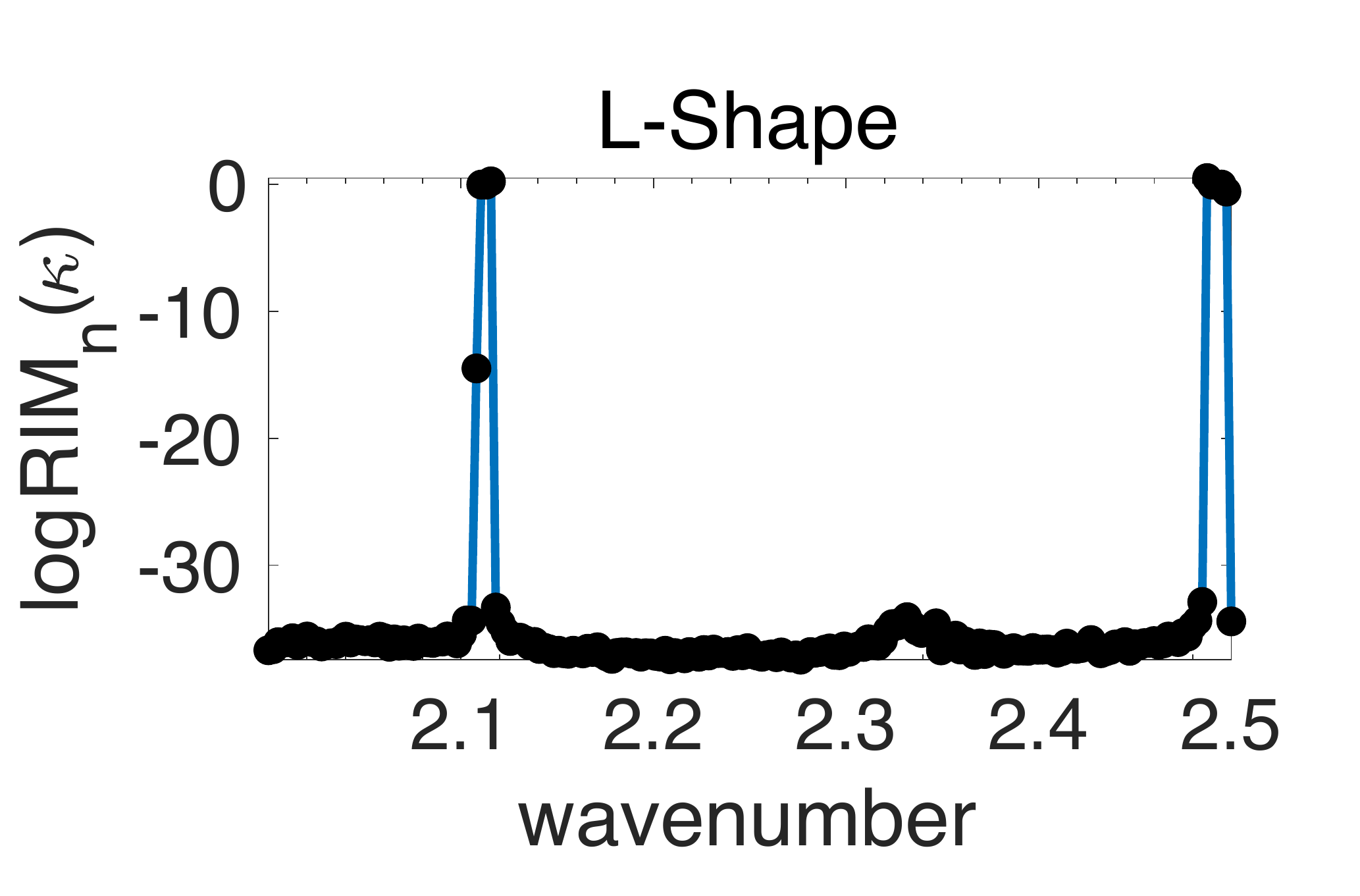}} \\
  \subfigure[]{\includegraphics[clip, trim = 0 220 0 230, width=.48\textwidth]{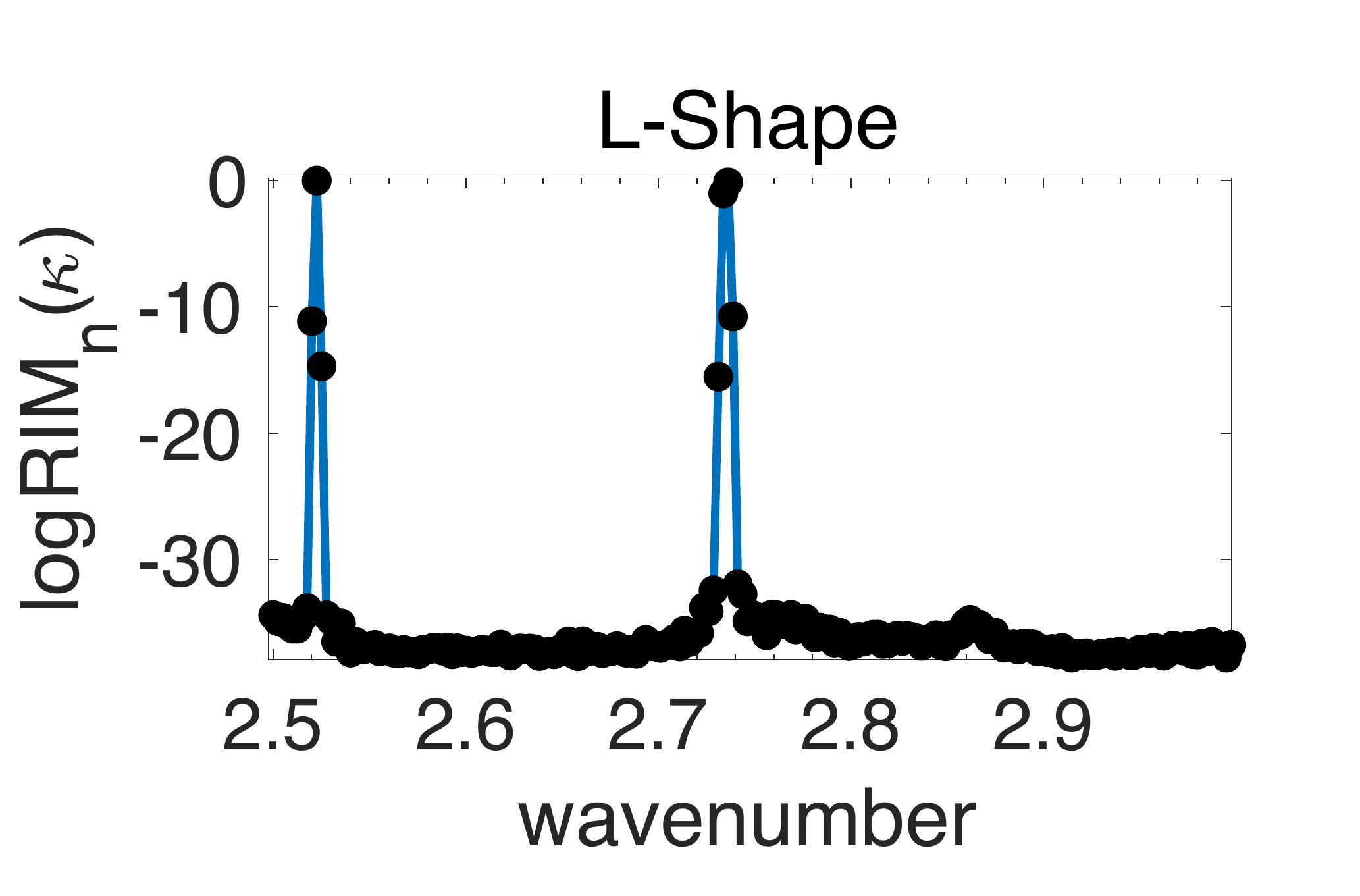}}
  \caption{Plots of $\log{{\rm  RIM}_n(\kappa)}$ for Example 5 in different intervals: (a) $[1.5, 2]$ (b) $[2, 2.5]$  (c) $[2.5, 3]$.} 
  \label{fig:L-shape} 
\end{figure}

\begin{example} 
In the last example, we consider a regular pentagon with vertexes $(\cos(\frac{2\pi j}{n}),\sin(\frac{2\pi j}{n})), j=0,1,2,3,4$. We choose the intervals to be $[1.5, 2]$, $[2, 2.5]$ and $[2.5, 3]$. Each interval is divided  into $200$ subintervals. The results are shown in Figure \ref{fig:pentagon}. The eigenvalues computed in this domain are $1.8945$, $2.2563$, $2.4724$, $2.6432$, $2.8769$, $2.9146$ and $2.9523$.
\end{example}

\begin{figure}
  \centering
  \subfigure[]{\includegraphics[clip, trim = 0 220 0 230, width=.48\textwidth]{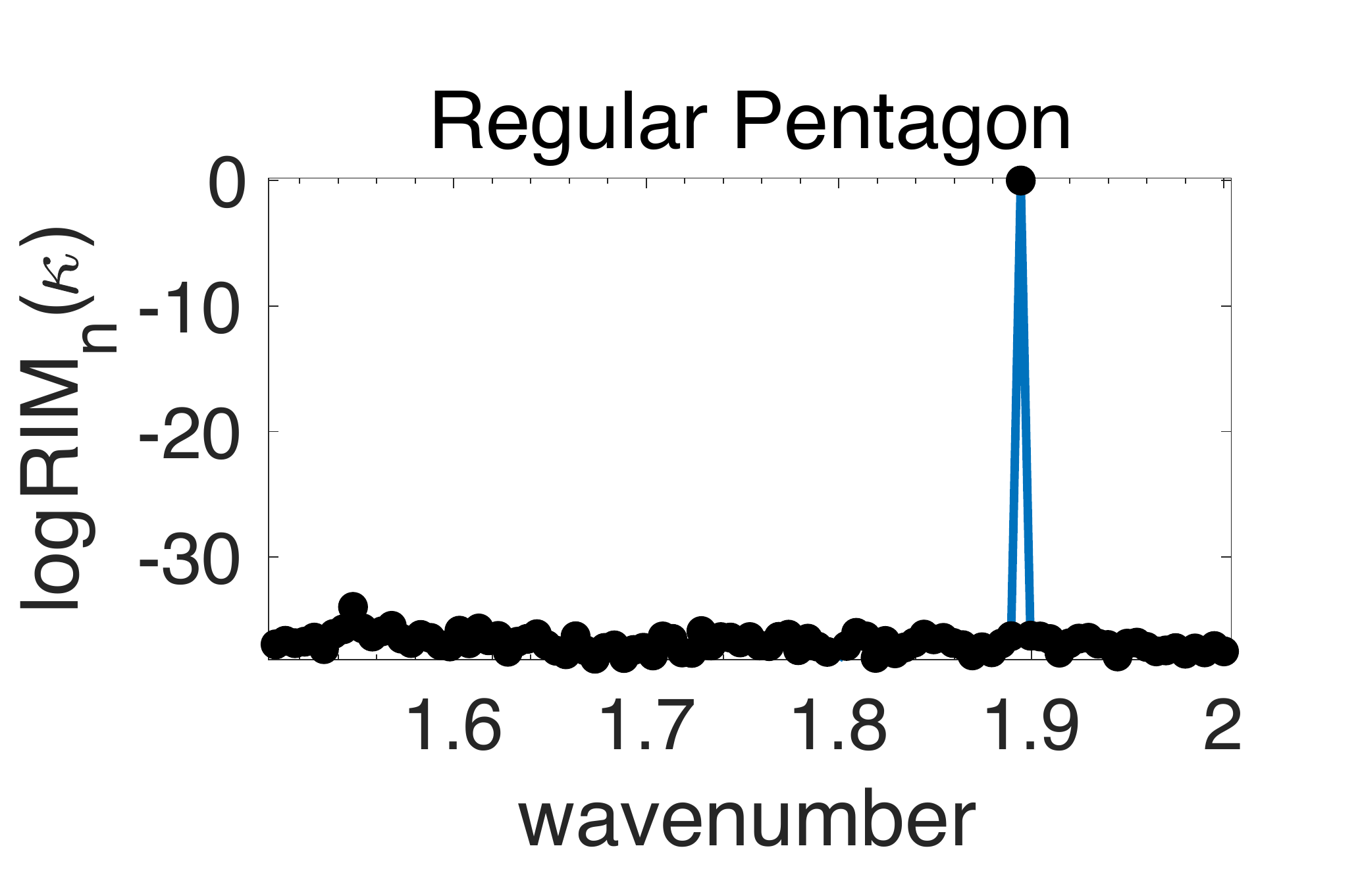}}
  \subfigure[]{\includegraphics[clip, trim = 0 220 0 230, width=.48\textwidth]{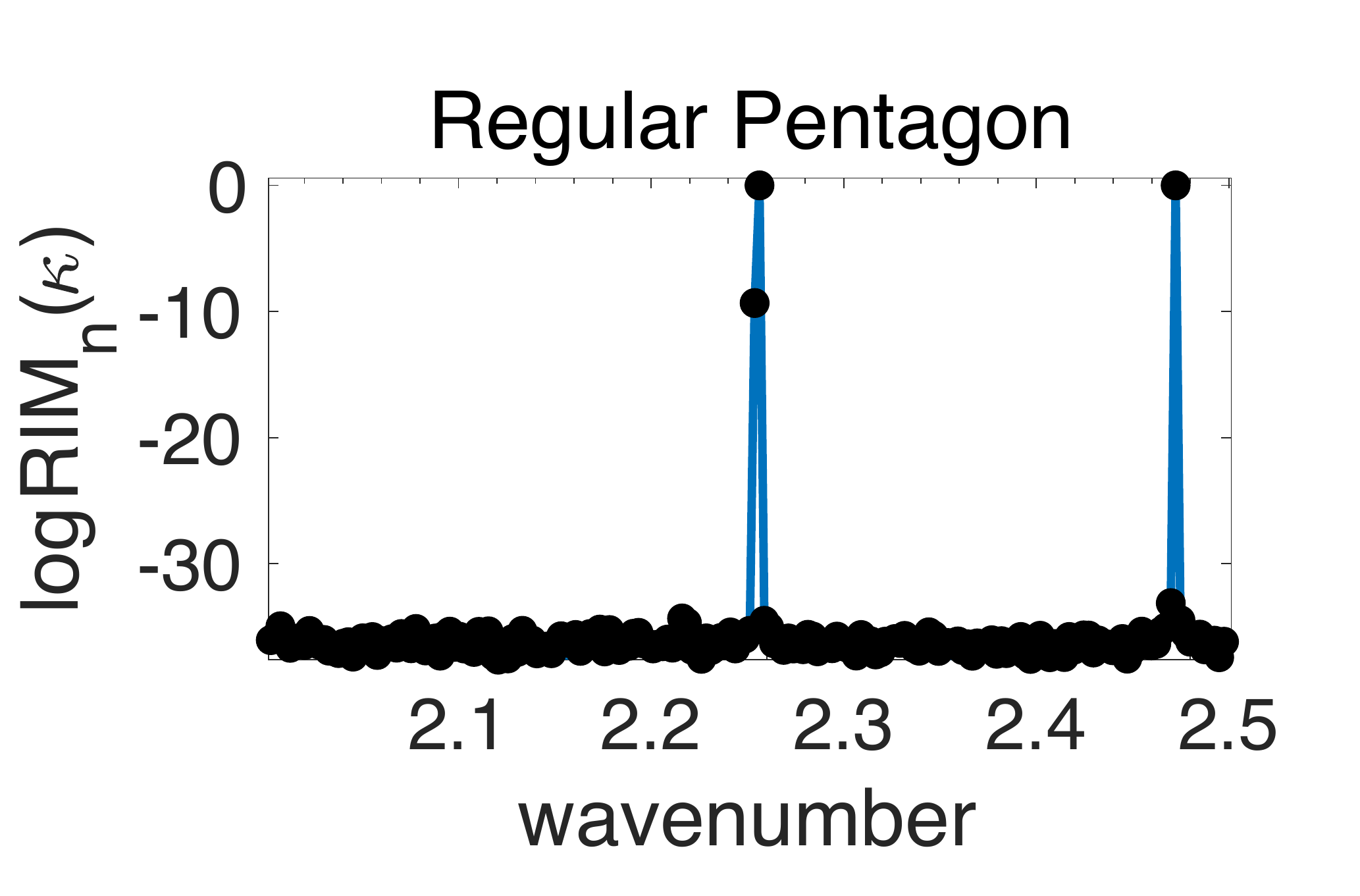}}
  \subfigure[]{\includegraphics[clip, trim = 0 220 0 230, width=.48\textwidth]{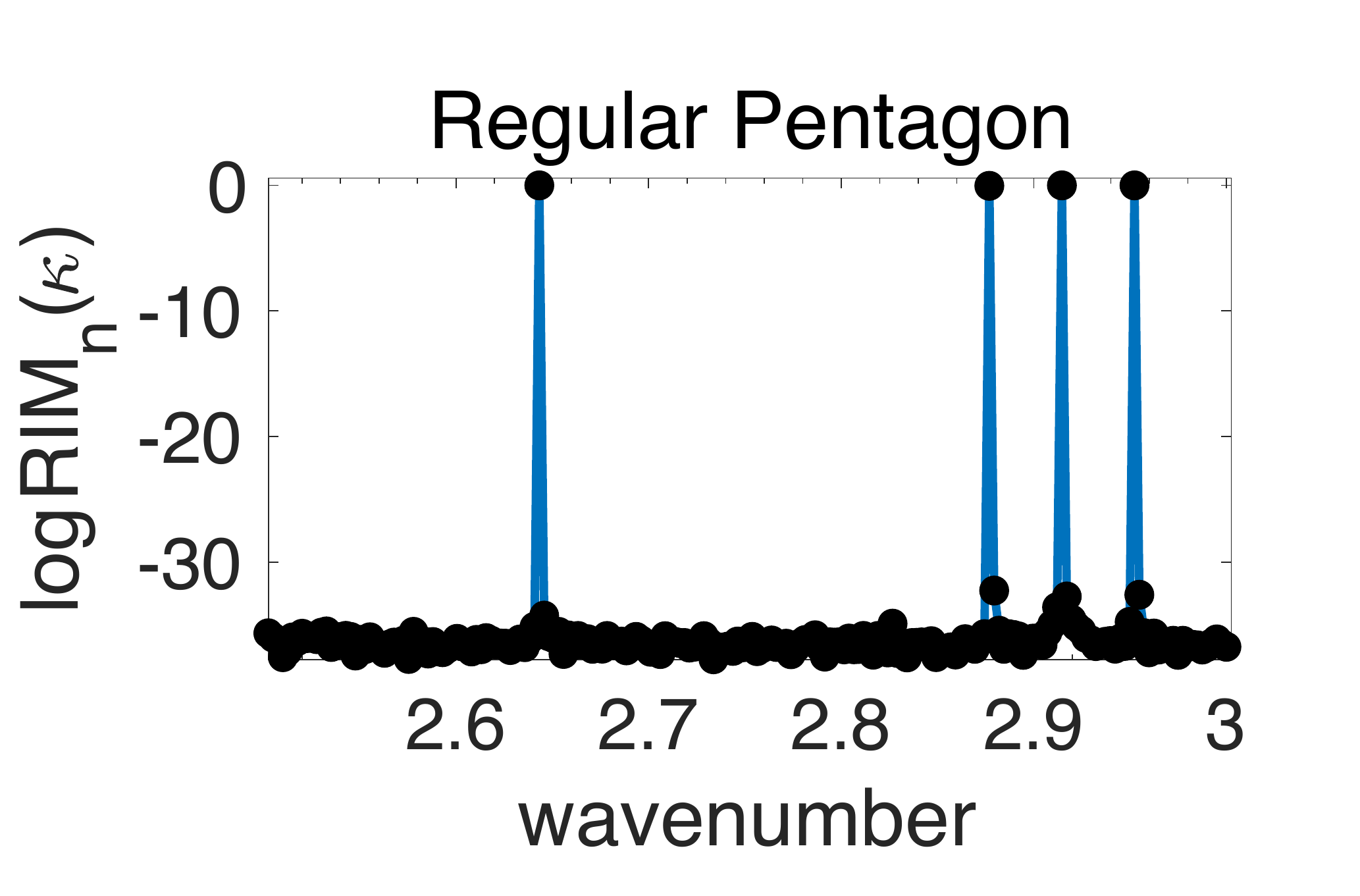}}
  \caption{Plots of $\log{{\rm  RIM}_n(\kappa)}$ for Example 6 in different intervals: (a) $[1.5, 2]$ (b)  $[2, 2.5]$ (c) $[2.5, 3]$.} 
  \label{fig:pentagon} 
\end{figure}



\section*{Acknowledgments}

The work of Y. Ma was supported by the NSFC grant No. 11901085 and the Research startup funds of DGUT No. GC300502-1. The work of F. Ma was supported by the NSFC grant No. 11771180. The work of Y. Guo was supported by the NSFC grant No. 11971133. The work of J. Li was partially supported by the NSFC grant No. 11971221, the Shenzhen Sci-Tech Fund No. JCYJ20190809150413261 and JCYJ20170818153840322 and Guangdong Provincial Key Laboratory of Computational Science and Material Design No. 2019B030301001. We would also like to thank Prof. Rainer Kress for his discussions on the Nystr\"{o}m method. 


\bibliographystyle{plain}
\bibliography{reference.bib}

\end{document}